\documentclass[12pt]{article}
\usepackage{fullpage}
\usepackage[english]{babel}
\usepackage{amsmath,amsthm,amssymb}
\usepackage{amsfonts}
\usepackage{graphicx}
\usepackage{epstopdf}
\usepackage{amsbsy}
\usepackage{algorithm}
\usepackage{subfig}
\usepackage{dsfont}
\usepackage{color}
\usepackage{mathrsfs}
\usepackage{bm}
\usepackage{bbm}
\usepackage[flushleft]{threeparttable}
\usepackage{relsize}
 
\newcommand{\N}{\mathbb{N}}

\newcommand{\R}{\mathbb{R}}
\newcommand{\E}{\mathbb{E}}

\usepackage{hyperref}
\usepackage{cleveref}

\newtheorem{theorem}{Theorem}[section]
\newtheorem{corollary}{Corollary}
\newtheorem{lemma}{Lemma}
\newtheorem{proposition}{Proposition}
\newtheorem{definition}{Definition}
\newtheorem{remark}{Remark}
\newtheorem{example}{Example}
\newtheorem{assumption}{Assumption}
\numberwithin{corollary}{section}
\numberwithin{lemma}{section}
\numberwithin{proposition}{section}
\numberwithin{definition}{section}
\numberwithin{remark}{section}
\numberwithin{equation}{section}
\numberwithin{table}{section}
\numberwithin{example}{section}
\numberwithin{assumption}{section}

\allowdisplaybreaks

\makeatletter
\let\save@mathaccent\mathaccent
\newcommand*\if@single[3]{%
  \setbox0\hbox{${\mathaccent"0362{#1}}^H$}%
  \setbox2\hbox{${\mathaccent"0362{\kern0pt#1}}^H$}%
  \ifdim\ht0=\ht2 #3\else #2\fi
  }
\newcommand*\rel@kern[1]{\kern#1\dimexpr\macc@kerna}
\newcommand*\widebar[1]{\@ifnextchar^{{\wide@bar{#1}{0}}}
{\wide@bar{#1}{1}}}
\newcommand*\wide@bar[2]{\if@single{#1}{\wide@bar@{#1}{#2}{1}}
{\wide@bar@{#1}{#2}{2}}}
\newcommand*\wide@bar@[3]{%
  \begingroup
  \def\mathaccent##1##2{%
    \let\mathaccent\save@mathaccent
    \if#32 \let\macc@nucleus\first@char \fi
    \setbox\z@\hbox{$\macc@style{\macc@nucleus}_{}$}%
    \setbox\tw@\hbox{$\macc@style{\macc@nucleus}{}_{}$}%
    \dimen@\wd\tw@
    \advance\dimen@-\wd\z@
    \divide\dimen@ 3
    \@tempdima\wd\tw@
    \advance\@tempdima-\scriptspace
    \divide\@tempdima 10
    \advance\dimen@-\@tempdima
    \ifdim\dimen@>\z@ \dimen@0pt\fi
    \rel@kern{0.6}\kern-\dimen@
    \if#31
      \overline{\rel@kern{-0.6}\kern\dimen@\macc@nucleus\rel@kern{0.4}
      \kern\dimen@}%
      \advance\dimen@0.4\dimexpr\macc@kerna
      \let\final@kern#2%
      \ifdim\dimen@<\z@ \let\final@kern1\fi
      \if\final@kern1 \kern-\dimen@\fi
    \else
      \overline{\rel@kern{-0.6}\kern\dimen@#1}%
    \fi
  }%
  \macc@depth\@ne
  \let\math@bgroup\@empty \let\math@egroup\macc@set@skewchar
  \mathsurround\z@ \frozen@everymath{\mathgroup\macc@group\relax}%
  \macc@set@skewchar\relax
  \let\mathaccentV\macc@nested@a
  \if#31
    \macc@nested@a\relax111{#1}%
  \else
    \def\gobble@till@marker##1\endmarker{}%
    \futurelet\first@char\gobble@till@marker#1\endmarker
    \ifcat\noexpand\first@char A\else
      \def\first@char{}%
    \fi
    \macc@nested@a\relax111{\first@char}%
  \fi
  \endgroup
}
\makeatother

\begin{document}

\title{A Bayesian Risk Approach to Data-driven Stochastic Optimization: Formulations and Asymptotics}
\author{Di Wu \quad Helin Zhu \quad Enlu Zhou\\
H. Milton Stewart School of Industrial and Systems Engineering,\\
Georgia Institute of Technology}
\date{}
\maketitle

\begin{abstract}
A large class of stochastic programs involve optimizing an expectation taken with respect to an underlying distribution that is unknown in practice. One popular approach to addressing the distributional uncertainty, known as the distributionally robust optimization (DRO), is to hedge against the worst case over an uncertainty set of candidate distributions. However, it has been observed that inappropriate construction of the uncertainty set can sometimes result in over-conservative solutions. To explore the middle ground between optimistically ignoring the distributional uncertainty and pessimistically fixating on the worst-case scenario, we propose a Bayesian risk optimization (BRO) framework for parametric underlying distributions, which is to optimize a risk functional applied to the posterior distribution of an unknown distribution parameter. Of our particular interest are four risk functionals: mean, mean-variance, value-at-risk, and conditional value-at-risk. To unravel the implication of BRO, we establish the consistency of objective functions and optimal solutions, as well as the asymptotic normality of objective functions and optimal values. More importantly, our analysis reveals a hidden interpretation: the objectives of BRO can be approximately viewed as a weighted sum of posterior mean objective and the (squared) half-width of the true objective's confidence interval.\\
\\
{\bf Keywords}: data-driven optimization, risk measures, Bayesian asymptotics.
\end{abstract}

\section{Introduction} \label{sec:1}
We consider a classical stochastic program
\begin{equation} \label{eq1.1}
\min_{x\in \mathcal{X}} \E_{\mathbb{P}^c}[h(x, \xi)],
\end{equation}
where $\mathcal{X}$ is a closed subset of $\R^d$ that implicitly incorporates any constraints, $\xi$ is a random vector taking value in $\R^m$, and $h$ is a function that maps $\R^d \times \R^m$ to $\R$. The expectation $\E$ is taken with respect to (w.r.t.) the distribution of $\xi$, denoted by $\mathbb{P}^c$. Although \cref{eq1.1} formulates a broad range of decision-making problems, it seems to overlook the fact that $\mathbb{P}^c$ is rarely known exactly in practice. More likely is that only an estimate $\hat{\mathbb{P}}$ can be obtained using finite real-world data. However, due to finite-sample error, even if the approximate problem
\begin{equation} \label{eq1.2}
\min_{x\in \mathcal{X}} \E_{\hat{\mathbb{P}}}[h(x, \xi)]
\end{equation}
is solved to optimality, the resulting solution may not perform quite as well under the true distribution $\mathbb{P}^c$. One popular approach to addressing this issue is to apply the framework of distributionally robust optimization (DRO), where one uses available data to construct an uncertainty/ambiguity set $\mathcal{D}$ that contains $\mathbb{P}^c$ with a high probability, and then optimize over $\mathcal{D}$ by hedging against the worst case, i.e.,
\begin{equation} \label{eq1.3}
\min_{x\in \mathcal{X}}\max_{\mathbb{P}\in \mathcal{D}} \E_{\mathbb{P}}[h(x, \xi)].
\end{equation}
An abundant literature exist on DRO and we refer the reader to \cite{bertsimas2013data,wiesemann2014distributionally,esfahani2015data,delage2010distributionally} for reviews and recent development. The key to the success of DRO is to construct a reasonable $\mathcal{D}$ such that \cref{eq1.3} is computationally tractable while maintaining certain performance guarantees. Two typical ways to construct $\mathcal{D}$ are based on: (i) distance metrics, such as $\phi$-divergence and Wasserstein distance (see e.g. \cite{bayraksan2015data,ben2013robust,jiang2015data,esfahani2015data,gao2016distributionally}); (ii) moment constraints (see e.g. \cite{scarf1958min,delage2010distributionally,popescu2007robust,zymler2013distributionally}). 
However, DRO's reliance on $\mathcal{D}$ can be a double-edged sword: as is observed in \cite{wang2016likelihood}, an inappropriately constructed $\mathcal{D}$ may lead to overly conservative solutions, i.e., solutions that perform poorly under far more realistic scenarios than the worst case.

The aforementioned drawbacks of \cref{eq1.2} and \cref{eq1.3} make us realize that there are a range of options from optimistically ignoring the distributional uncertainty to pessimistically fixating on the worst case. In this paper, we explore the middle ground by taking a Bayesian approach towards the uncertainty in $\mathbb{P}^c$. Suppose $\mathbb{P}^c$ belongs to a parametric family $\{\mathbb{P}_{\theta} \mid \theta \in \Theta\}$, where $\Theta$ is the parameter space and $\theta^c \in \Theta$ is the unknown true parameter. 
Through a Bayesian perspective, $\theta^c$ can be viewed as a realization of a belief random variable $\tilde{\theta}$, whose posterior distribution can be computed based on given data. Informally, we propose to solve the problem
\begin{equation} \label{eq1.4}
\min_{x\in \mathcal{X}} \rho_{\, \mathbb{P}_n}\left\{\E_{\mathbb{P}_\theta}[h(x, \xi)]\right\}.
\end{equation}
In \cref{eq1.4}, $\rho$ is a risk functional applied to $\E_{\mathbb{P}_\theta}[h(x, \xi)]$, which is a random variable induced by  $\theta \sim \mathbb{P}_n$, where $ \mathbb{P}_n$ is the posterior distribution of $\tilde{\theta}$. We will refer to \cref{eq1.4} as the Bayesian risk optimization (BRO) problem. The framework of BRO first appeared in our preliminary work \cite{zhou2015simulation}, in which four choices of $\rho$ were considered: mean, mean-variance, value-at-risk, and conditional value-at-risk. Intuitively speaking, the posterior serves as a natural indicator of the uncertainty about $\theta^c$, thus playing a similar role to that of the $\mathcal{D}$ in DRO. A formal and detailed introduction to BRO is deferred to \cref{sec:2.2}.

The research on solving the BRO problem is still an ongoing progress. In some cases, the structure of $\rho$ can be exploited to develop efficient computational methods. For instance, \cite{zhu2016simulation} proposes a Monte Carlo method to solve the conditional value-at-risk formulation by adaptively adjusting the risk level. Although tractability is of great practical interest to study, understanding the implications of BRO is at least as important as solving it. To date, there has only been empirical evidence suggesting the possible benefits of considering such an optimization problem. It is therefore the primary goal of this paper to develop a deeper understanding of BRO. Specifically, we achieve this by conducting a standard investigation of its asymptotic behaviors. Focusing on the aforementioned four choices of $\rho$, we establish the consistency of objective functions and optimal solutions, as well as the asymptotic normality of objective functions and optimal values. More importantly, our analysis reveals a hidden interpretation: \emph{the objectives of BRO can be approximately viewed as a weighted sum of posterior mean objective and the (squared) half-width of the true objective's confidence interval.} Interestingly, similar insight has also been developed for DRO in \cite{gotoh2015robust}, which shows that a large class of robust empirical optimization problems are essentially equivalent to a mean-variance formulation. In the same spirit, \cite{lam2016robust} showed that the robust sensitivity of an expectation with respect to the unknown distribution can be decomposed as the mean plus a term depending on the standard deviation. Nevertheless, due to the many differences between DRO and BRO, a direct comparison between these two formulations is difficult, if not impossible. Instead, our work aims to provide a different perspective on approaching distributional uncertainty, which hopefully will add one more option to the toolbox of practitioners.

The statistical properties of stochastic programs have been examined under a variety of contexts. For example, \cite{shapiro1991asymptotic} discusses a general approach to studying the asymptotics of statistical estimators in stochastic programming, and \cite{kleywegt2002sample} investigates the asymptotic properties of optimal values and solutions for the sample average approximation problem. Notably, \cite{dentcheva2016statistical} also establishes central limit theorems for composite risk functionals, and discusses the asymptotic behavior of stochastic programs with objectives being composite risk functionals. Nonetheless, aside from the difference in settings (frequentist vs. Bayesian), the major distinctions between \cite{dentcheva2016statistical} and our work lie in the distributions to which risk functionals are applied and the associated proof techniques. On the one hand, \cite{dentcheva2016statistical} considers a class of risk functionals applied to a sequence of empirical distributions, where a version of uniform Central Limit Theorem and an extended Delta Theorem can be applied to show the weak convergence of risk functionals. On the other hand, we apply risk functionals to a sequence of posterior distributions, where a Bayesian Central Limit Theorem guarantees that the total variation distance between the posterior distribution and a normal distribution (with a random mean) vanishes in probability (see \cref{lemma4.2}). The theorem's intricate form adds more technicalities and subtleties to our analysis. In particular, a given risk functional is not necessarily uniformly continuous relative to the total variation metric. In addition, while \cite{dentcheva2016statistical} studies the asymptotic properties of risk functional as an estimator of a ``true'' functional, we study similar asymptotics in an effort to uncover the implications of solving the proposed BRO problem.


The rest of the paper is organized as follows. We explicitly construct an underlying probability space and provide a formal introduction to BRO in \cref{sec:2}. \Cref{sec:3,sec:4} are devoted to establishing consistency and asymptotic normality results related to BRO, respectively. We reveal the hidden implication of solving BRO in \cref{sec:5}, and finally conclude in \cref{sec:6}.

\section{Preliminaries} \label{sec:2}
We begin by restating the problem settings and formalizing Bayesian framework in a measure-theoretic context. Then, a formal introduction to BRO is given along with a discussion on its connections with DRO.

\subsection{Construction of probability space} \label{sec:2.1}
We consider a case where $\mathbb{P}^c$, i.e., the true underlying distribution of $\xi$ in \cref{eq1.1}, belongs to a parametric family of distributions $\{\mathbb{P}_{\theta} \mid \theta \in \Theta\}$, where $\Theta$ is the parameter space. In particular, this encompasses distributions with a finite support, where the probability mass vector can be treated as a finite-dimensional parameter. Suppose the form of $\mathbb{P}_{\theta}$ is known but the true parameter $\theta^c$ is not. Through a Bayesian perspective, we view $\theta^c$ as a realization of a belief random variable $\tilde{\theta}$. Denote by $\pi$ the prior distribution of $\tilde{\theta}$.  Also assume that we have a dataset $\{\xi_i\}_{i=1}^n$, which, conditioned on $\tilde{\theta}$, are $n$ independent and identically distributed (i.i.d.) samples with distribution $\mathbb{P}_{\tilde{\theta}}$. To perform a rigorous analysis of BRO's asymptotic properties, we explicitly construct a probability space $(\Omega, \mathcal{F}, \mu)$ such that (i) both $\tilde{\theta}$ and $\xi_i$ are random variables defined on this space; (ii) $\tilde{\theta}$ follows the prior distribution $\pi$; (iii) conditioned on $\tilde{\theta}$,  $\{\xi_i\}$ are i.i.d. samples from $\mathbb{P}_{\tilde{\theta}}$. Our construction follows the standard approach in Bayesian literature (e.g., \cite{schwartz1965bayes}).

Suppose that $\tilde{\theta}$ takes value in a parameter space $\Theta \subset \R^l$ equipped with a Borel $\sigma$-algebra $\mathcal{B}_\Theta$ and a probability measure $\pi$, while $\xi_i$ takes value in $\Xi \subset \R^m$ equipped with a Borel $\sigma$-algebra $\mathcal{B}_{\Xi}$ and a collection of probability measures $\{\mathbb{P}_{\theta}\}_{\theta \in \Theta}$. Then, the probability space induced by $n$ i.i.d. copies of $\xi_i$ is $(\Xi^n, \mathcal{B}_{\Xi}^n,  \mathbb{P}^n_{\theta})$, where $\mathcal{B}_{\Xi}^n$ is the product $\sigma$-algebra $\mathcal{B}_{\Xi} \otimes \cdots \otimes \mathcal{B}_{\Xi}$ \footnote{For two $\sigma$-algebras $\Sigma_1$ and $\Sigma_2$, $\Sigma_1 \otimes \Sigma_2$ is defined as $\sigma\left(\Sigma_1 \times \Sigma_2 \right)$.} and $\mathbb{P}^n_{\theta}$ is the product measure $\mathbb{P}_{\theta} \times \cdots \times \mathbb{P}_{\theta}$. Next, we apply Kolmogorov's Extension Theorem (see, e.g., \cite{durrett2010probability}, Theorem A.3.1) to extend $(\Xi^n, \mathcal{B}_{\Xi}^n,  \mathbb{P} ^n_{\theta})$ to a sequence space $(\Xi^{\N}, \mathcal{B}^{\N}_{\Xi}, \mathbb{P}^{\N}_{\theta})$, where $\Xi^{\N}$ is the space of all infinite sequences in $\Xi$, and $\mathcal{B}^{\N}_{\Xi}$ is the $\sigma$-algebra generated by all cylinder sets of the form
$$\left\{\bar{\xi} \in \Xi^{\N} \mid (\bar{\xi}_1, \bar{\xi}_2, \ldots, \bar{\xi}_n) \in B \right\}, \quad B \in \mathcal{B}^n_{\Xi},$$
where $\bar{\xi}_i$ is the $i$th entry of the sequence $\bar{\xi}$. Correspondingly, $\mathbb{P}^{\N}_{\theta}$ is the product measure that coincides with $\mathbb{P}^n_{\theta}$ on $\mathcal{B}^n_{\Xi}$, i.e.,
$$\mathbb{P}^{\N}_{\theta} \left(\left\{\bar{\xi} \in \Xi^{\N} \mid (\bar{\xi}_1, \bar{\xi}_2, \ldots, \bar{\xi}_n) \in B\right\} \right) = \mathbb{P}^n_{\theta}(B),  \quad \forall B \in \mathcal{B}^n_{\Xi}.$$
Let $\Omega = \Theta \times \Xi^{\N}$ be the sample space equipped with the $\sigma$-algebra $\mathcal{F} = \mathcal{B}_{\Theta} \otimes \mathcal{B}^{\N}_{\Xi}$. We assume that the following holds throughout the paper.

\begin{assumption} \label{assump1}
 $\mathbb{P}_{\theta}(B)$ is a measurable function of $\theta$ for all $B \in \mathcal{B}^{\N}_{\Xi}$.
\end{assumption}

Under \cref{assump1}, we define a measure $\mu$ on the collection of all rectangle sets in $\mathcal{F}$ with the following property.
\begin{equation} \label{eq2.1}
\mu(A \times B) := \int_A \int_B \mathbb{P}^{\N}_{\theta}(d\xi) \pi(d\theta) = \int_A \mathbb{P}^{\N}_{\theta}(B) \pi(d\theta), \quad \forall A\in \mathcal{B}_{\Theta}, B \in \mathcal{B}^{\N}_{\Xi}.
\end{equation}
The integrals in \cref{eq2.1} are well-defined by \cref{assump1}. Moreover, since the rectangle sets form a semialgebra (see, e.g., \cite{durrett2010probability}, page 3), there exists a unique extension of $\mu$ to $\mathcal{F}$ by Carath\'eodory's Extension Theorem (see, e.g., \cite{durrett2010probability}, Theorem A.1.3).
\begin{proposition}\label{prop2.1}
For any $C \in \mathcal{F}$, we have $\mu(C) = \int_{\Theta} \mathbb{P}^{\N}_{\theta} (C_\theta) \pi(d\theta)$, where $C_{\theta}$ is defined as $\left\{\bar{\xi} \in \Xi^{\N}: (\theta, \bar{\xi}) \in C\right\}.$
\end{proposition}
\begin{proof}
Let $\mathcal{L} := \left\{C \in \mathcal{F} \mid \mu(C) = \int_{\Theta} \mathbb{P}^{\N}_{\theta} (C_\theta) \pi(d\theta)\right\}$ and $\mathcal{P}: = \{A\times B\mid A \in \Theta, B \in \Xi^{\N}\}$. Then $\mathcal{L}$ is a $\lambda$-system and $\mathcal{P}$ is a $\pi$-system. From \cref{eq2.1} we know $\mathcal{P} \subseteq \mathcal{L}$, so $\mathcal{F} = \sigma(\mathcal{P}) \subseteq \mathcal{L}$ by Dynkin's $\pi-\lambda$ Theorem (see, e.g., \cite{durrett2010probability}, Theorem A.1.4).
\end{proof}

\begin{remark} \label{remark1}
\cref{prop2.1} characterizes $\mu$ for all sets in $\mathcal{F}$. Loosely speaking, $\mu$ is the joint distribution of $\tilde{\theta}$ and an infinite sequence $\{\xi_i\}$ that is i.i.d. conditioned on $\tilde{\theta}$. An important observation is that if $C \in \mathcal{F}$ has $\mu(C)=1$, then $\mathbb{P}_{\theta}(C_{\theta}) = 1$ for all $\theta \in \Theta$ up to a set of measure 0 under $\pi$. This particular mode of convergence is due to $1 = \mu(C) = \int_{\Theta} \mathbb{P}_{\theta}(C_{\theta}) \pi(d\theta)$.
\end{remark}

It remains to define $\tilde{\theta}$ and $\xi_i$ as random variables on $(\Omega, \mathcal{F}, \mu)$. Take $\omega \in \Omega$ and write $\omega$ as $(\omega_{\theta}, \omega_{\xi})$ such that $\omega_{\theta} \in \Theta$ and $\omega_{\xi} \in \Xi^{\N}$. Define $\tilde{\theta}(\omega) := \omega_{\theta}$ and $\xi_i(\omega) := (\omega_{\xi})_i$, where $(\cdot)_i$ extracts the $i$th entry of a sequence. Under this type of construction, every realization $\omega$ yields a parameter $\tilde{\theta}$ and an infinite sequence $\{\xi_i\}$. Furthermore, it can be verified that conditioned on $\tilde{\theta}$, $\{\xi_i\}$ are i.i.d. samples from $\mathbb{P}_{\tilde{\theta}}$. The Bayes estimator (under a quadratic loss function) can be expressed by $\E_{\mu}[\tilde{\theta} \mid \mathcal{F}_n]$, where $\mathcal{F}_n := \sigma(\xi_1, \ldots, \xi_n)$ is the filtration generated by data, and the posterior distribution of $\tilde{\theta}$ is given by $\mathbb{P}_n(\cdot):= \mathbb{\mu}(\{\omega \in\Omega \mid \tilde{\theta}(\omega) \in \cdot\} \mid \mathcal{F}_n)$. 

\subsection{Bayesian risk optimization} \label{sec:2.2}
Under the parametric assumption on $\mathbb{P}^c$ in \cref{sec:2.1}, we define the following function for notational brevity.
\begin{equation} \label{eq2.2}
H(x, \theta) : = \E_{\mathbb{P}_\theta}[h(x, \xi)].
\end{equation}

Then \cref{eq1.1} and \cref{eq1.3} can be rewritten as
\begin{equation} \label{eq2.3}
\min_{x\in \mathcal{X}} H(x, \theta^c) \qquad \text{and} \qquad \min_{x\in \mathcal{X}} \max_{\theta \in \tilde{\Theta}} H(x, \theta),
\end{equation}
respectively, where $\tilde{\Theta}$ is a subset of $\Theta$. We assume that $H$ is finite for every pair $(x, \theta) \in \mathcal{X} \times \Theta$. Notice that a well-designed $\tilde{\Theta}$ in \cref{eq2.3} should reflect the level of uncertainty in the data to infer $\theta^c$. For example, it is preferable if the diameter of $\tilde{\Theta}$ shrinks as more data are observed. Meanwhile, we know from Bayesian asymptotic theory that the posterior converges weakly to a point mass on $\theta^c$ at an exponential rate \cite{shen2001rates}. This motivates the idea of using the posterior distribution of $\tilde{\theta}$ to capture the degree of parameter uncertainty, which leads to the following BRO problem.
\begin{equation} \label{eq2.4}
\min_{x \in \mathcal{X}} \rho_{\, \mathbb{P}_n} \left[H(x, \theta) \right].
\end{equation}
In \cref{eq2.4}, $\rho$ is a risk functional which is defined as a mapping from a random variable to a real number, and $\mathbb{P}_n$ is the posterior distribution of $\tilde{\theta}$ given $n$ data samples. In particular, a risk functional that is subadditive, monotonically increasing, positive homogenous and translation-invariant is called a \emph{coherent risk measure}. We refer the reader to \cite{artzner1999coherent,rockafellar2013fundamental,shapiro2014lectures} and the references therein for an axiomatic definition of coherent risk measures and its related discussions. Numerous choices of $\rho$ can be applied to \cref{eq2.4}. We follow \cite{zhou2015simulation} and investigate the following four of them in this paper.
\begin{enumerate}
\item
The mean and mean-variance formulations:
\begin{equation*}
\min_{x \in \mathcal{X}} \E_{\mathbb{P}_n}\left[H(x, \theta) \right] + w \text{Var}_{\mathbb{P}_n}\left[H(x, \theta) \right], \quad w\ge0.
\end{equation*}
\item
The value-at-risk (VaR) formulation:
\begin{equation*}
\min_{x \in \mathcal{X}} \text{VaR}_{\mathbb{P}_n}^{\alpha}\left[H(x, \theta) \right], \quad \alpha \in (0,1).
\end{equation*}
\item
The conditional value-at-risk (CVaR) formulation:
\begin{equation*}
\min_{x \in \mathcal{X}} \text{CVaR}_{\mathbb{P}_n}^{\alpha}\left[H(x, \theta) \right], \quad \alpha \in (0,1).
\end{equation*}
\end{enumerate}

In particular, VaR and CVaR are two commonly used risk measures in financial engineering for controlling large loss. For a random variable $X$, $\text{VaR}^{\alpha} (X)$ is defined as the $\alpha$-quantile of $X$, i.e.,
\begin{equation*}
\text{VaR}^{\alpha}(X) := \inf \{t: \mathbb{P}(X \le t) \ge \alpha \},
\end{equation*}
and CVaR is defined as the expected loss beyond VaR,
\begin{equation*}
\text{CVaR}^{\alpha}(X) := \frac{1}{1-\alpha} \int_\alpha^1 \text{VaR}^r(X) dr.
\end{equation*}
When $\mathbb{P}(X = \text{VaR}^\alpha(X)) = 0$, CVaR can also be written as a conditional expectation,
\begin{equation} \label{eq2.5}
\text{CVaR}^{\alpha}(X) := \E[X \mid X \ge \text{VaR}^\alpha(X)].
\end{equation}

A risk functional is called \emph{law-invariant} if it depends only on the distribution of the random variable. We remark that all four choices of $\rho$ considered here are law-invariant. Furthermore, mean and CVaR are coherent risk measures; VaR is a risk measure but is not coherent because it is not subadditive; mean-variance is not a risk measure for its lack of monotonicity.

Some connections between BRO and DRO are drawn as follows. First, coherent risk measures can be represented as optimization problems using duality theory (see, e.g, \cite{shapiro2014lectures}, Section 6.3), which allows for a DRO interpretation in terms of ambiguity sets. Second, it is possible to reformulate a BRO as a DRO problem. For example, let $\rho$ be VaR with a risk level $\alpha$. Suppose that  $H$ is continuous on $\mathcal{X} \times \Theta$, $\Theta$ is compact and $\mathbb{P}_n$ has a positive density on $\Theta$, then for $\alpha = 100\%$, \cref{eq2.4} can be rewritten as
\begin{equation} \label{eq2.6}
\min_{x\in \mathcal{X}}\text{VaR}^{100\%}_{\mathbb{P}_n} \left[H(x, \theta)\right] = \min_{x\in \mathcal{X}}\max_{\theta\in \Theta}H(x, \theta),
\end{equation}
where the right hand side (RHS) corresponds to DRO with $\Theta$ being viewed as an ambiguity set of $\theta^c$. It can also be observed that by adjusting the risk level $\alpha$, the VaR objective can easily accommodate a wide range of risk preferences from being overly optimistic to being highly risk-averse.

We highlight a few main results before proceeding to the proofs. Let $\mathcal{N}$ denote a normal distribution, and let $\phi$ and $\Phi$ denote $\mathcal{N}(0,1)$'s density and cumulative distribution function, respectively. We use ``$\Rightarrow$'' to denote weak convergence (see \cref{def3.2}). The following results are in the pointwise sense, i.e., they hold for every fixed $x \in \mathcal{X}$. Specifically, as $n \rightarrow \infty$,
\begin{enumerate}
\item[(i)]
for the mean and mean-variance objectives,
\begin{equation*}
\sqrt{n}\left\{\E_{\mathbb{P}_n}[H(x,\theta)] + w\text{Var}_{\mathbb{P}_n}[H(x,\theta)] - H(x,\theta^c) \right\} \Rightarrow \mathcal{N}\left(0, \sigma_x^2\right);
\end{equation*}

\item[(ii)]
for the VaR objective,
\begin{equation*}
\sqrt{n}\left\{\text{VaR}_{\mathbb{P}_n}^\alpha[H(x,\theta)] - H(x,\theta^c) \right\} \Rightarrow \mathcal{N}(\sigma_x \Phi^{-1}(\alpha), \sigma_x^2);
\end{equation*}

\item[(iii)]
for the CVaR objective,
\begin{equation*}
\sqrt{n} \left\{\text{CVaR}_{\mathbb{P}_n}^\alpha [H(x,\theta)] - H(x,\theta^c) \right\} \Rightarrow \mathcal{N}\left(\frac{\sigma_x}{1-\alpha}\phi(\Phi^{-1}(\alpha)), \sigma_x^2 \right).
\end{equation*}
\end{enumerate}
In (i)-(iii), the limiting variance in the RHS is defined as
$$\sigma_x^2: = \nabla_\theta H(x, \theta^c)^\intercal [I(\theta^c)]^{-1} \nabla_\theta H(x, \theta^c),$$
where $\nabla_\theta H(x,\theta^c)$ is the gradient of $H(x, \cdot)$ at $\theta^c$, the superscript``$\intercal$'' stands for transpose, and $I(\theta^c)$ is the Fisher information that $\xi_i$ carries about $\theta^c$. An immediate consequence of (i)-(iii) is that confidence intervals (CIs) can be constructed for $H(x, \theta^c)$, which is the true objective value. More importantly, as we will show in \cref{sec:5}, these results imply that the objectives of BRO problems are approximately equivalent to a weighted sum of posterior mean objective and the (squared) half-width of the true objective's CI. In other words, BRO essentially seeks to balance the trade-off between posterior mean performance and the robustness in actual performance.

\section{Consistency of Bayesian risk optimization} \label{sec:3}
Since the distributional uncertainty diminishes as $n \rightarrow \infty$, one naturally expects the objectives of BRO problems to recover the true objective $H(\cdot, \theta^c)$, and the optimal solutions of BRO problems to converge to the true optimal solutions.  Let $D_{KL}(P \| Q)$ denote the K-L divergence between two distributions $P$ and $Q$, and let ``a.s.'' be short for ``almost surely''. The following assumption is made to guarantee the strong consistency of posterior distribution.
\begin{assumption}[Sufficient conditions for consistency under $\mathbb{P}^{\N}_{\theta^c}$] \label{assump3.1}
\quad
\begin{enumerate}
\item[(i)]
$\Theta$ is a compact set.
\item[(ii)]
For all $n$, $\mathbb{P}^n_{\theta}(\cdot)$ \footnote{Recall from \cref{sec:2.1} that $\mathbb{P}^n_{\theta}$ is defined as the product measure of $n$ copies of $\mathbb{P}_\theta$.} has a density $p^n_{\theta}(\cdot)$ that is $\mathcal{B}_{\Theta} \otimes \mathcal{B}_{\Xi}$-measurable.
\item[(iii)]
For any neighborhood $V \in \mathcal{B}_{\Theta}$ of $\theta^c$, there exists a sequence of uniformly consistent tests of the hypothesis $\tilde{\theta} = \theta^c$ against the alternative $\tilde{\theta} \in \Theta \setminus V$ \footnote{This condition implies separability of $\theta^c$ from $\Theta \setminus V$. For more details on uniformly consistent tests, we refer the reader to \cite{schwartz1965bayes}.} .
\item[(iv)]
For any $\epsilon>0$ and any neighborhood $V \in \mathcal{B}_{\Theta}$ of $\theta^c$, $V$ contains a subset $W$ such that $\pi(W)>0$ and $D_{KL}\left(p_{\theta^c} \| p_{\theta}\right)<\epsilon$ for all $\theta \in W$.
\end{enumerate}
\end{assumption}

\begin{lemma}[\cite{schwartz1965bayes}, Theorem 6.1] \label{lemma3.1}
Suppose \cref{assump3.1} holds. Then for any neighborhood $V \in \mathcal{B}_{\Theta}$ of $\theta^c$, $\mathbb{P}_n(V) \rightarrow 1$ a.s. $(\mathbb{P}^{\N}_{\theta^c})$ as $n \rightarrow \infty$.
\end{lemma}

\subsection{Consistency of objective functions} \label{sec3.1}
The pointwise weak consistency of BRO problems' objectives has been shown in \cite{zhou2015simulation}. We strengthen this result by proving the pointwise strong consistency of BRO problems' objectives, where the proof technique differs from that in \cite{zhou2015simulation}. Moreover, our result is essential to establishing the consistency of optimal solutions.

\begin{definition}[Weak convergence] \label{def3.2}
A sequence of random variables $\{X_n\}$ is said to converge weakly (or in distribution) to $X$, denoted by $X_n \Rightarrow X$, if and only if $\E[g(X_n)] \rightarrow \E[g(X)]$ as $n \rightarrow \infty$ for all $g$ bounded and continuous. Similarly, a sequence of distributions $\bar{\mathbb{P}}_n\Rightarrow \bar{P}$ if and only if $\int g(\omega) \bar{\mathbb{P}}_n(d\omega) \rightarrow  \int g(\omega) \bar{\mathbb{P}}(d\omega)$ as $n \rightarrow \infty$ for all $g$ bounded and continuous.
\end{definition}

\begin{lemma} \label{lemma3.3}
Suppose \cref{assump3.1} holds. Then $\mathbb{P}_n \Rightarrow \delta_{\theta^c}$ a.s. $(\mathbb{P}^{\N}_{\theta^c})$, where $\delta_{\theta^c}$ is a point mass on $\theta^c$.
\end{lemma}
\begin{proof}
Let $\Theta_m \subseteq \Theta$ be an open ball centered at $\theta^c$ with radius $1/m$. \cref{lemma3.1} ensures that for each $\Theta_m$ there exists an event $\Omega_m \in \mathcal{B}_{\Xi}^{\N}$ with $\mathbb{P}^{\N}_{\theta^c}(\Omega_m) = 1$ such that $\mathbb{P}_n(\Theta_m) \rightarrow 1$ as $n\rightarrow \infty$ on $\Omega_m$. Define $\tilde{\Omega} : = \cap_{m=1}^\infty \Omega_m$, then $\mathbb{P}^{\N}_{\theta^c}(\tilde{\Omega}) = 1$ and it suffices to show $\mathbb{P}_n \Rightarrow \delta_{\theta^c}$ on $\tilde{\Omega}$. Take a sample path $\omega \in \tilde{\Omega}$. Notice that for any bounded and continuous function $g$ and a fixed positive integer $k$, $\mathbb{P}_n(\Theta \setminus \Theta_k) \rightarrow 0$ and thus $ \int_{\Theta \setminus \Theta_k} g(\theta) \mathbb{P}_n(d\theta)  \rightarrow 0$ as $n \rightarrow \infty$. It follows that
\begin{equation*}
\inf_{\theta \in \Theta_k} g(\theta) \le \liminf_{n \rightarrow \infty} \int_{\Theta} g(\theta) \mathbb{P}_n(d\theta) \le \limsup_{n \rightarrow \infty} \int_{\Theta} g(\theta) \mathbb{P}_n(d\theta) \le \sup_{\theta \in \Theta_k} g(\theta).
\end{equation*}
Letting $k \rightarrow \infty$, the continuity of $g$ and \cref{def3.2} implies that $\mathbb{P}_n \Rightarrow \delta_{\theta^c}$.
\end{proof}

\begin{theorem} \label{thm3.4}
Suppose \cref{assump3.1} holds, and $H(x, \cdot)$ is continuous on $\Theta$ for every $x \in \mathcal{X}$. Then for every fixed $x\in \mathcal{X}$, we have $\rho_{\, \mathbb{P}_n}[H(x,\theta)] \rightarrow H(x,\theta^c)$ as $n \rightarrow \infty$ a.s. $(\mathbb{P}^{\N}_{\theta^c})$ for all four choices of $\rho$.
\end{theorem}
\begin{proof}
Suppress $x$ and write $H(x, \theta)$ as $H(\theta)$ for short. We will focus on the same event $\tilde{\Omega}$ constructed in the proof of \cref{lemma3.3}. Take a sample path $\omega \in \tilde{\Omega}$. The consistency for each choice of $\rho$ is shown as follows.

{\bf{Mean.}} The compactness of $\Theta$ and the continuity of $H$ implies that $H$ is bounded on $\Theta$. It follows directly from \cref{def3.2} that $\E_{\mathbb{P}_n}[H(\theta)] \rightarrow H(\theta^c)$.

{\bf{Mean-variance.}} Since $\text{Var}_{\mathbb{P}_n} [H(\theta)]  = \E_{\mathbb{P}_n}\{[H(\theta)]^2\} - \{\E_{\mathbb{P}_n}[H(\theta)]\}^2$, where $H^2$ and $H$ are bounded and continuous functions, it follows from \cref{def3.2} that $\text{Var}_{\mathbb{P}_n} [H(\theta)] \rightarrow [H(\theta^c)]^2 -  [H(\theta^c)]^2 = 0$.

{\bf{VaR.}} Let $\mathbb{P}^n_H := \mathbb{P}_n \circ H^{-1}$ be the distribution of $H(\theta)$ induced by $\mathbb{P}_n$. Then $\mathbb{P}^n_H \Rightarrow \delta_{H(\theta^c)}$ by Continuous Mapping Theorem (see, e.g., \cite{durrett2010probability}, Theorem 3.2.4). Since $H(\theta^c) - \epsilon$ and $H(\theta^c) + \epsilon$ are continuity points of $\delta_{H(\theta^c)}$, we have for any $\epsilon > 0$ that 
\begin{equation*}
\mathbb{P}^n_H\left(H(\theta) \le H(\theta^c) - \epsilon\right) \rightarrow 0 \le \alpha, \qquad \mathbb{P}^n_H\left(H(\theta) \le H(\theta^c) + \epsilon\right) \rightarrow 1 \ge \alpha,
\end{equation*}
which implies that $H(\theta^c) - \epsilon \le \text{VaR}^{\alpha}_{\mathbb{P}_n}[H(\theta)] \le H(\theta^c) + \epsilon$ for all $n$ sufficiently large. The convergence follows from that $\epsilon$ can be chosen arbitrarily small.

{\bf{CVaR.}} By CVaR's translation invariance and monotonicity,
\begin{align*}
\big|\text{CVaR}^\alpha_{\mathbb{P}_n}[H(\theta)] - H(\theta^c)\big| &\le \text{CVaR}^\alpha_{\mathbb{P}_n} \left[\big|H(\theta) - H(\theta^c) \big| \right]\\
& = \frac{1}{1-\alpha} \E_{\mathbb{P}_n}[\big|H(\theta) - H(\theta^c) \big| \mathbbm{1}_{\{{|H(\theta) - H(\theta^c) | \ge v_{\alpha}}\}}]\\
&\le \frac{1}{1-\alpha} \E_{\mathbb{P}_n}[\big|H(\theta) - H(\theta^c) \big|],
\end{align*}
where $\mathbbm{1}(\cdot)$ is an indicator function and $v_{\alpha} := \text{VaR}^{\alpha}_{\mathbb{P}_n}[|H(\theta) - H(\theta^c)|]$. The proof is complete by noting that $|H(\cdot) - H(\theta^c)|$ is bounded and continuous on $\Theta$.
\end{proof}

\subsection{Consistency of optimal solutions} \label{sec3.2}
Let $S_n : = \arg\min_{x \in \mathcal{X}} \rho_{\, \mathbb{P}_n} [H(x, \theta)]$ be the set of optimal solutions of a BRO problem, and let $S: = \arg\min_{x\in \mathcal{X}} H(x, \theta)$ be the set of true optimal solutions. We consider the following deviation between $S_n$ and $S$.
\begin{definition} \label{def3.5}
For $A, B \subseteq \mathcal{X}$, define $\mathbb{D}(A,B) := \sup_{x \in A} \text{dist}(x, B)$, where $\text{dist}(x, B) : = \inf_{y \in B} \|x - y\|$ and $\|\cdot\|$ denotes an arbitrary norm.
\end{definition}

The Hausdorff metric is defined as $\max\{\mathbb{D}(A,B), \mathbb{D}(B,A)\}$, but it suffices for us to consider $\mathbb{D}$.  We will assume that $\mathcal{X}$ is compact, which is not a strong assumption since the optimal solutions are often contained in a compact set. 
\begin{assumption}[Sufficient conditions for consistency under $\mu$] \label{assump3.2}
\quad
\begin{enumerate}
\item[(i)]
$\{\Xi, \mathcal{B}_{\Xi}\}$ and $\{\Theta, \mathcal{B}_{\Theta}\}$ are both isomorphic to Borel sets in a complete separable space.
\item[(ii)]
If $\theta_1 \neq \theta_2$, then there exists a set $A\in \mathcal{B}^{\N}_{\Xi}$ for which $\mathbb{P}_{\theta_1}(A) \neq \mathbb{P}_{\theta_2}(A)$.
\end{enumerate}
\end{assumption}

By Doob's Consistency Theorem \cite{doob1949application}, \cref{assump3.2} implies that for any neighborhood $V \in \mathcal{B}_{\Theta}$ of $\theta^c$, $\mathbb{P}_n(V) \rightarrow 1$ $\text{a.s.}$ $(\mu)$ as $n \rightarrow \infty$, which is weaker than \cref{lemma3.1} (see \cref{remark1} for a comparison between $\mu$ and $\mathbb{P}^{\N}_{\theta^c}$) because \cref{assump3.2} is significantly less stringent than \cref{assump3.1}. However, notice that working with measure $\mu$ allows an expression of posterior mean as a conditional expectation, where Martingale Convergence Theorem can be applied. The following lemmas will be useful in showing that $\mathbb{D}(S_n, S) \rightarrow 0$ a.s. ($\mu$) as $n \rightarrow \infty$.

\begin{lemma}[\cite{shapiro2014lectures}, Theorem 5.3]  \label{lemma3.6}
Let $\mathcal{X}$ be a compact subset of $\R^d$. Suppose a sequence of continuous functions $\{f_n\} : \mathcal{X} \rightarrow \R$ converges uniformly to a continuous function $f$. Let $\bar{S}_n := \arg\min_{x \in \mathcal{X}} f_n(x)$ and $\bar{S}:=\arg\min_{x \in \mathcal{X}} f(x)$. Then $\mathbb{D}(\bar{S}_n, \bar{S}) \rightarrow 0$ as $n \rightarrow \infty$. Furthermore, we have $f^*_n \rightarrow f^*$, where $f^*_n := \min_{x \in \mathcal{X}} f_n(x)$ and $f^*: = \min_{x \in \mathcal{X}} f(x)$.
\end{lemma}

\begin{lemma}[\cite{thomson2008elementary}, Exercise 9.4.10] \label{lemma3.7}
Let $\mathcal{X}$ be a compact subset of $\R^d$. If $\{f_n\} : \mathcal{X} \rightarrow \R$ is a sequence of functions converging pointwise to a function $f$, and there exists a common Lipschitz constant $L>0$ for $\{f_n\}$ and $f$, then $f_n \rightarrow f$ uniformly.
\end{lemma}

\begin{lemma} \label{lemma3.8}
Let $X, Y$ be two random variables in $\mathcal{L}^\infty(\Omega, \mathcal{F}, \mathbb{P})$, i.e., the space of all essentially bounded random variables. For $\rho$ a risk functional with monotonicity and translation invariance, $|\rho(X) - \rho(Y)| \le \|X-Y\|_{\infty}$, where $\|\cdot\|_{\infty}$ is the $L^{\infty}$ norm. If furthermore $\rho$ is a coherent risk measure, then $|\rho(X) - \rho(Y)| \le \rho(|X-Y|)$.
\end{lemma}
\begin{proof}
See Lemma 4.3 in \cite{hans2002stochastic} for proof of the first part. For the second part, by subadditivity $\rho(Y) + \rho(X-Y) \ge \rho(X)$, and by monotonicity $\rho(X) - \rho(Y) \le \rho(X-Y) \le \rho(|X-Y|)$. The result follows by symmetry.
\end{proof}

\begin{theorem} \label{thm3.9}
Suppose that \cref{assump3.2} holds, $\Theta$ and $\mathcal{X}$ are compact, and $H(x,\cdot)$ is continuous on $\Theta$ for every $x \in \mathcal{X}$. Then, $\mathbb{D}(S_n, S) \rightarrow 0$ a.s. $(\mu)$ as $n\rightarrow \infty$ if (i) for mean and CVaR, there exists a measurable function $\kappa: \Theta \rightarrow \R^+$ with $|H(x, \theta) - H(y, \theta)| \le \kappa(\theta) \|x-y\|, \forall x,y \in \mathcal{X}$ and $\int_{\Theta} \kappa(\theta) \pi(d\theta) < \infty$; (ii) for mean-variance, $H$ is jointly continuous on $\mathcal{X} \times \Theta$; (iii) for VaR, (i) holds with $\|\kappa\|_\infty < \infty$.
\end{theorem}
\begin{proof}
The following argument is in the sense of a.s. ($\mu$). Similar to the proof of \cref{thm3.4}, it can be shown that $\rho_{\, \mathbb{P}_n}[H(\cdot, \theta)] \rightarrow H(\cdot, \theta^c)$ pointwise on $\mathcal{X}$ as $n \rightarrow \infty$. If we further show that $\rho_{\, \mathbb{P}_n}[H(\cdot, \theta)]$ has a common Lipschitz constant $L$ for all $n$, then \cref{lemma3.7} implies that $\rho_{\, \mathbb{P}_n}[H(\cdot, \theta)] \rightarrow H(\cdot, \theta^c)$ uniformly on $\mathcal{X}$, and $\mathbb{D}(S_n, S) \rightarrow 0$ is an immediate consequence of the first part of \cref{lemma3.6}.

{\bf Mean.} Recall from \cref{sec:2.1} that the posterior mean can be expressed as a conditional expectation. Thus, we have for all $x, y \in \mathcal{X}$,
\begin{align*}
&\big| \E_{\mathbb{P}_n}[H(x, \theta)] - \E_{\mathbb{P}_n}[H(y, \theta)] \big| \le  \E_\mu[\kappa(\theta) \mid \mathcal{F}_n] \|x-y\|.
\end{align*}
By assumption $\E_\mu[|\kappa(\theta)|] = \int \kappa(\theta) \pi(d\theta) < \infty$, so $\E_\mu[\kappa(\theta) \mid \mathcal{F}_n]$ is a Doob martingale and by Martingale Convergence Theorem (see, e.g., \cite{durrett2010probability}, Theorem 5.5.7),
$$\E_\mu[\kappa(\theta) \mid \mathcal{F}_n] \rightarrow \E_\mu[\kappa(\theta) \mid \mathcal{F}_\infty] \quad \text{as } n \rightarrow \infty,$$
where $\mathcal{F}_{\infty}:= \sigma\left(\cup_n \mathcal{F}_n \right)$. Since $\E_\mu\{\E_\mu[\kappa(\theta) \mid \mathcal{F}_{\infty}] \} = \E_\mu[\kappa(\theta)] < \infty$, $\E_\mu[\kappa(\theta) \mid \mathcal{F}_{\infty}]$ is a.s. finite, and there exists an $L:= \sup_{n} \E_\mu[\kappa(\theta) \mid \mathcal{F}_n] < \infty$.

{\bf Mean-variance.} It suffices to find an $L$ for $\text{Var}_{\mathbb{P}_n}[H(\cdot, \theta)]$. By definition,
\begin{equation} \label{eq3.1}
\text{Var}_{\mathbb{P}_n}[H(\cdot, \theta)] = \E_{\mathbb{P}_n}\left\{[H(\cdot, \theta)]^2\right\} - \left\{\E_{\mathbb{P}_n}[H(\cdot, \theta)] \right\}^2.
\end{equation}
Since $H$ is jointly continuous on $\mathcal{X} \times \Theta$, $|H| \le M$ for some $M \ge 0$. For the first term in the RHS of \cref{eq3.1},
\begin{align*}
\E_{\mathbb{P}_n}\left\{[H(x, \theta)]^2 - [H(y,\theta)]^2 \right\} &\le \E_{\mathbb{P}_n} \left\{ \big| H(x, \theta) + H(y, \theta) \big| \cdot  \big| H(x, \theta) - H(y, \theta) \big| \right\}\\
& \le 2M  \E_{\mathbb{P}_n} \left\{ \big| H(x, \theta) - H(y, \theta) \big| \right\}.
\end{align*}
Similarly, for the second term,
\begin{align*}
&\big| \left\{\E_{\mathbb{P}_n}[H(x, \theta)] \right\}^2 - \left\{\E_{\mathbb{P}_n}[H(y, \theta)] \right\}^2 \big|\\
 \le & \:\big| \E_{\mathbb{P}_n}[H(x, \theta)] + \E_{\mathbb{P}_n}[H(y, \theta)] \big| \cdot \big| \E_{\mathbb{P}_n}[H(x, \theta)] - \E_{\mathbb{P}_n}[H(y, \theta)] \big| \\
 \le & \: 2M \big| \E_{\mathbb{P}_n}[H(x, \theta)] - \E_{\mathbb{P}_n}[H(y, \theta)] \big|,
\end{align*}
and the rest follows from the case of mean.

{\bf VaR.} Since $\|\kappa\|_\infty < \infty$, there exists $L>0$ such that $\big| H(x, \theta) - H(y, \theta) \big| \le L \|x-y\|$ for all $x,y \in \mathcal{X}$ and $\theta \in \Theta$. By the first part of \cref{lemma3.8},
\begin{equation*}
\big| \text{VaR}^{\alpha}_{\mathbb{P}_n}[H(x, \theta)] - \text{VaR}^{\alpha}_{\mathbb{P}_n}[H(y, \theta)] \big|  \le \|H(x, \theta) - H(y, \theta)\|_{\infty} \le L\|x-y\|.
\end{equation*}

{\bf CVaR.} Using the second part of \cref{lemma3.8}, we have for any $x, y$ in $\mathcal{X}$,
\begin{equation*}
\big| \text{CVaR}^{\alpha}_{\mathbb{P}_n}[H(x, \theta)] - \text{CVaR}^{\alpha}_{\mathbb{P}_n}[H(y, \theta)]\big| \le \frac{1}{1-\alpha} \E_\mu[\kappa(\theta) \mid \mathcal{F}_n] \|x-y\|,
\end{equation*}
and the rest follows from the proof of the mean formulation.
\end{proof}

\begin{corollary}
For all four choices of $\rho$,
$$\min_{x \in \mathcal{X}} \rho_{\, \mathbb{P}_n}[H(x,\theta)] \rightarrow \min_{x\in \mathcal{X}} H(x, \theta^c) \quad \text{a.s. } (\mu).$$
\end{corollary}
\begin{proof}
Since we have established the uniform convergence of BRO problems' objectives in the proof of \cref{thm3.9}, the result is due to the second part of \cref{lemma3.6}.
\end{proof}

As a special case, convex functions have the following nice property regarding uniform convergence: if a sequence of convex functions converges pointwise on an open set $O\subset \R^n$, then it also converges uniformly on any compact subset of $O$ (see, e.g., \cite{hiriart2012fundamentals}, Theorem 3.1.4). This leads to the following corollary of \cref{thm3.4} for convex risk measures (e.g., mean and CVaR).

\begin{corollary} \label{cor3.11}
Suppose \cref{assump3.1} holds, and $H(x, \cdot)$ is continuous on $\Theta$ for every $x \in \mathcal{X}$. Also let $H(\cdot, \theta)$ be convex in $x$ for all $\theta \in \Theta$, then for the mean and the CVaR formulations,  we have $\mathbb{D}(S_n, S) \rightarrow 0$ a.s. $(\mathbb{P}^{\N}_{\theta^c})$.
\end{corollary}

\section{Asymptotic normality of objectives} \label{sec:4}
We present two types of asymptotic normality results in this section. First, we show for a fixed $x$ that $\sqrt{n}\{\rho_{\, \mathbb{P}_n}[H(x, \theta)] - H(x, \theta^c)\}$ converges weakly to a normal distribution. Then, we extend this result by establishing weak convergence of $\sqrt{n}\{\rho_{\, \mathbb{P}_n}[H(\cdot, \theta)] - H(\cdot, \theta^c)\}$ in the space of continuous functions. To begin with, define $Z_n(\theta) := \sqrt{n} (\theta - \theta^c)$ and let $\mathbb{P}_{Z_n} := \mathbb{P}_n \circ Z_n^{-1}$ be the distribution of $Z_n$ induced by $\mathbb{P}_n$. 
\begin{definition}
For two probability measures $\mu$ and $\nu$ on a measurable space $(\Omega, \mathcal{F})$, their total variation distance is defined as $\|\mu - \nu\|_{\text{TV}} := \sup_{A \in \mathcal{F}} |\mu(A)-\nu(A)| $.
\end{definition}

\begin{lemma}[Bernstein-von Mises Theorem] \label{lemma4.2}
Under mild conditions, 
\begin{equation} \label{eq4.1}
\big\|\mathbb{P}_{Z_n} - \mathcal{N}(\Delta_{n}, [I(\theta^c)]^{-1})\big\|_{\emph{TV}} \rightarrow 0 \text{ in probability }(\mathbb{P}^{\N}_{\theta^c}) \text{ as } n \rightarrow \infty,
\end{equation}
where $\mathcal{N}$ denotes a normal distribution, $I(\theta^c)$ is the Fisher information $\xi_i$ carries about $\theta^c$, and $\Delta_n \Rightarrow \mathcal{N}(0, [I(\theta^c)]^{-1})$ as $n \rightarrow \infty$.
\end{lemma}

We refer the reader to Theorem 10.1 in \cite{van2000asymptotic} for detailed conditions of \cref{lemma4.2}, which are mild and are assumed to hold in all subsequent proofs. We remark that \cref{lemma4.2} is also commonly referred to as the Bayesian Central Limit Theorem. Recall that we consider a law-invariant $\rho$, so there is no ambiguity in writing $\rho(\mathbb{P})$ for some distribution $\mathbb{P}$. The forthcoming proofs of asymptotic normality are motivated by the following heuristic argument. 

{\bf Step 1.} If $\rho$ is translation-invariant and positive homogeneous, then 
\begin{equation} \label{gap1}
\sqrt{n} \{\rho_{\, \mathbb{P}_n}[H(x, \theta)] - H(x, \theta^c) \} = \rho_{\, \mathbb{P}_n} \{\sqrt{n} [H(x, \theta) - H(x, \theta^c)]\} \approx \rho_{\, \mathbb{P}_n}[X_n(\theta)],
\end{equation}
where $X_n(\theta): =\nabla_\theta H(x, \theta^c)^\intercal Z_n(\theta)$ is the first-order Taylor approximation.

{\bf Step 2.} Based on \cref{lemma4.2}, show that 
\begin{equation} \label{gap2}
\|\mathbb{P}_n \circ X_n^{-1} - \mathcal{N}(\nabla_\theta H(x, \theta^c)^\intercal\Delta_n, \sigma^2_x)\|_{\text{TV}} \rightarrow 0 \quad \text{in probability,}
\end{equation}
where $\sigma^2_x := \nabla_\theta H(x, \theta^c)^\intercal [I(\theta^c)]^{-1} \nabla_\theta H(x, \theta^c)$.

{\bf Step 3.} 
Since $\rho_{\, \mathbb{P}_n}(X_n(\theta)) = \rho(\mathbb{P}_n \circ X_n^{-1})$, it suffices to show that
\begin{equation} \label{gap3}
\rho(\mathbb{P}_n \circ X_n^{-1}) - \rho[\mathcal{N}(\nabla_\theta H(x, \theta^c)^\intercal\Delta_n, \sigma^2_x)] \rightarrow 0 \quad \text{in probability.}
\end{equation}

If the above argument holds, then the asymptotic distribution of BRO problems' objectives can be easily characterized since $\rho[\mathcal{N}(\nabla_\theta H(x, \theta^c)^\intercal\Delta_n, \sigma^2_x)]$ allows closed forms for all four choices of $\rho$. However, each step listed above involves a gap to be closed. In particular, note that $\mathcal{N}(\nabla_\theta H(x, \theta^c)^\intercal\Delta_n, \sigma^2_x)$ is not a fixed measure. Thus, from a general perspective, step 3 essentially investigates the following: for two sequences of probability measures $\{\mu_n\}$ and $\{\nu_n\}$, does $\|\mu_n - \nu_n \|_{\text{TV}} \rightarrow 0$ imply that $\rho(\mu_n) - \rho(\nu_n) \rightarrow 0$? In other words, when $\rho$ is viewed as a functional of distributions, is it uniformly continuous relative to the total variation metric? Unfortunately, this is not true for our four choices of $\rho$. Nevertheless, it is possible for us to exploit the structure of $\mathcal{N}(\Delta_n, [I(\theta^c)]^{-1})$ to circumvent this issue.

\subsection{Asymptotic normality at a fixed $x$} \label{sec:4.1}
Once $x$ is fixed, we write $H(x, \theta)$ as $H(\theta)$ for notational brevity. Consider $\|\cdot\|$ being the Euclidean norm henceforth for convenience. In establishing asymptotic normality, each choice of $\rho$ has distinct properties and deserves separate treatment. To bridge the gaps in the preceding sketch of proof, we need the following regularity condition.

\begin{assumption} \label{assump4.1}
There exists a constant $\gamma>0$ such that for all $\epsilon >0$, there exists an $M_{\epsilon} >0$ satisfying
$$\mathbb{P}^{\N}_{\theta^c} \left\{\E_{\mathbb{P}_n} \left[\|\sqrt{n}(\theta - \theta^c)\|^{1+\gamma}\right] > M_{\epsilon} \right\} < \epsilon, \quad \forall n.$$
\end{assumption}

\cref{assump4.1} can be viewed as an ``in probability'' version of uniform integrability, because on event $\{\E_{\mathbb{P}_n} \left[\|\sqrt{n}(\theta - \theta^c)\|^{1+\gamma}\right] \le M_{\epsilon}\}$, we have 
\begin{equation} \label{eqUI}
\E_{\mathbb{P}_n} \left[\|\sqrt{n}(\theta - \theta^c)\| \mathbbm{1}_{\{\|\sqrt{n} (\theta - \theta^c)\| > K\}}\right] \le \frac{\E_{\mathbb{P}_n} \left[\|\sqrt{n}(\theta - \theta^c)\|^{1+\gamma}\right]}{K^\gamma} \le \frac{M_{\epsilon}}{K^\gamma}
\end{equation}
for any $K>0$. Thus, for sufficiently large $K$, the truncated tail expectation of $\|\sqrt{n} (\theta - \theta^c)\|$ can be arbitrarily small with a large probability ($\mathbb{P}^{\N}_{\theta^c}$) for all $n$. Another implication is that
$$\mathbb{P}_n(\|\sqrt{n} (\theta - \theta^c)\| > K) \le \frac{M_{\epsilon}}{K^{1+\gamma}}$$
by Markov's inequality. As we will see in the proof of \cref{thm4.4}, \cref{assump4.1} plays a vital role in bounding the remainder term in Taylor expansion. The following lemma is a special case of Theorem 10.8 in \cite{van2000asymptotic}, where the conditions are implicitly assumed to hold in all subsequent proofs. 
\begin{lemma} \label{lemma4.3}
Under mild assumptions,
\begin{equation*}
\sqrt{n}\left(\E_{\mathbb{P}_n}[\theta] - \theta^c\right) \Rightarrow \mathcal{N}(0, [I(\theta^c)]^{-1}), \quad \text{as } n \rightarrow \infty.
\end{equation*}
\end{lemma}

We now verify \cref{assump4.1} for some commonly used conjugate priors. Notice that if $\Theta \subseteq \R$, then for $\gamma = 1$,
\begin{equation*}
\E_{\mathbb{P}_n}\left\{[\sqrt{n}(\theta - \theta^c)]^2\right\} = \left\{\sqrt{n} (\E_{\mathbb{P}_n}[\theta] - \theta^c)\right\}^2 + n \text{Var}_{\mathbb{P}_n}[\theta],
\end{equation*}
where the first term in the RHS converges in distribution by \cref{lemma4.3}, so we only need to check if the second term is bounded in probability $(\mathbb{P}^{\N}_{\theta^c})$.
\begin{example}
Let $\xi_i \sim \text{Expo}(\theta^c)$ and $\pi \sim \text{Gamma}(\alpha_0, \beta_0)$. Then, ${\mathbb{P}_n}$ is given by $\text{Gamma}(\alpha_n, \beta_n)$, where $\alpha_n = \alpha_0 + n$ and $\beta_n = \beta_0 + \sum_{i=1}^n \xi_i$. Furthermore,
\begin{equation*}
n \text{Var}_{\mathbb{P}_n}[\theta] = \frac{n \alpha_n}{{\beta_n}^2} = \left(\frac{n}{\beta_0 + \sum_{i=1}^n \xi_i} \right)^2 + \frac{\alpha_0 n}{(\beta_0 + \sum_{i=1}^n \xi_i)^2} \rightarrow (\theta^c)^2 \quad \text{a.s. } (\mathbb{P}^{\N}_{\theta^c} ),
\end{equation*}
where the convergence follows from the strong law of large numbers (SLLN).
\end{example}

\begin{example}
Let $\xi_i \sim \mathcal{N}(\theta^c, \sigma^2)$, where $\sigma^2$ is known and $\pi \sim \mathcal{N}(\mu_0, \sigma_0^2)$. We then have
\begin{equation*}
n \text{Var}_{\mathbb{P}_n}[\theta] = n \sigma_n^2 = \frac{n}{1/\sigma_0^2 + n/\sigma^2} \rightarrow \sigma^2 \quad  \text{a.s. } (\mathbb{P}^{\N}_{\theta^c}).
\end{equation*}
\end{example}

\begin{example}
Let $\xi_i \sim \text{Weibull}(\theta^c, \beta)$, where $\theta^c$ is an unknown scale parameter and $\beta$ is a known shape parameter. Let the posterior of $\tilde{\theta}^\beta$ be $\text{InvGamma}(\alpha_n, \beta_n)$, where $\alpha_n = \alpha_0 + n$ and $\beta_n = \beta_0 + \sum_{i=1}^n \xi_i^\beta$. Then by the SLLN, 
\begin{equation*}
n \text{Var}_{\mathbb{P}_n}[\theta^\beta] = \frac{n\beta_n^2}{(\alpha_n-1)^2(\alpha_n-2)}  = \frac{n^3 (\beta_0/n + \sum_{i=1}^n \xi_i^\beta /n)^2}{(\alpha_0+n-1)^2 (\alpha_0 + n-2)} \rightarrow (\theta^c)^{2\beta} \text{a.s. } (\mathbb{P}^{\N}_{\theta^c}).
\end{equation*}
\end{example}

\begin{example}
Let $\xi_i$ be a discrete random variable supported on $\{y_1, \ldots, y_l\}$. Suppose $\mathbb{P}(\xi_i = y_i) = \theta^c_i$, then $\theta^c:=(\theta^c_1, \ldots, \theta^c_l)$ can be viewed as a parameter in $\R^l$. Choose $\pi \sim \text{Dirichlet}(\alpha_0)$, where $\alpha_0 = (1,\ldots, 1)$. It follows that $\mathbb{P}_n \sim \text{Dirichlet}(\alpha_n)$, where $\alpha_n = \alpha_0 + (N_1, \ldots, N_l)$ and $N_i := \sum_{j=1}^n \mathbbm{1}_{\{\xi_j = y_i\}}$. Let $\theta_i$, $\theta^c_i$ and $\alpha^n_i$ denote the $i$th component of $\theta$, $\theta^c$ and $\alpha_n$, respectively. Since
\begin{equation*}
\E_{\mathbb{P}_n}\left[\|\sqrt{n} (\theta - \theta^c)\|^2\right] = n\sum_{i=1}^l \E_{\mathbb{P}_n}[(\theta_i - \theta^c_i)^2],
\end{equation*}
it suffices to check the convergence of $n\E_{\mathbb{P}_n}[(\theta_i - \theta^c_i)^2]$ for each $i$. Likewise,
\begin{equation*}
n\E_{\mathbb{P}_n}[(\theta_i - \theta^c_i)^2] = n(\E_{\mathbb{P}_n}[\theta_i] - \theta^c_i)^2 + n\text{Var}_{\mathbb{P}_n}(\theta_i),
\end{equation*}
where, by noting $\sum_{j=1}^l \alpha^n_j = l + n$, we have
\begin{equation*}
\sqrt{n}(\E_{\mathbb{P}_n}[\theta_i] - \theta^c_i) = \sqrt{n}\left(\frac{\alpha^n_i}{l+n} - \theta^c_i \right) = \sqrt{n} \left(\frac{1 + N_i}{l+n} - \theta^c_i \right),
\end{equation*}
which converges weakly by the Central Limit Theorem, and the SLLN implies that
\begin{equation*}
n\text{Var}_{\mathbb{P}_n}(\theta_i) = n \frac{\alpha^n_i(l+n - \alpha^n_i)}{(l+n)^2(l+n+1)} \rightarrow \theta^c_i (1 - \theta^c_i) \quad \text{a.s. } (\mathbb{P}^{\N}_{\theta^c}).
\end{equation*}
\end{example}

The proofs of asymptotic normality will be presented in the order of mean, mean-variance, VaR and CVaR. Recall from previous notation that
$$\sigma^2_x := \nabla_\theta H(x, \theta^c)^\intercal [I(\theta^c)]^{-1} \nabla_\theta H(x, \theta^c).$$

\begin{theorem} \label{thm4.4}
Let \cref{assump3.1} and \cref{assump4.1} hold. If $H$ is continuous on $\Theta$ and differentiable at $\theta^c$, then 
\begin{equation*}
\sqrt{n}\left\{\E_{\mathbb{P}_n}[H(\theta)] - H(\theta^c) \right\} \Rightarrow \mathcal{N}\left(0, \sigma^2_x\right) \quad \text{as } n\rightarrow \infty.
\end{equation*}
If furthermore \cref{assump4.1} holds with $\gamma=1$, then as $n\rightarrow \infty$,
\begin{equation*}
\sqrt{n}\left\{\E_{\mathbb{P}_n}[H(\theta)] + w\text{Var}_{\mathbb{P}_n}[H(\theta)] - H(\theta^c) \right\} \Rightarrow \mathcal{N}\left(0, \sigma^2_x \right).
\end{equation*}
\end{theorem}
\begin{proof}
The first-order Taylor expansion of $H$ around $\theta^c$ yields
\begin{equation} \label{eq4.5}
\E_{\mathbb{P}_n}[\sqrt{n}(H(\theta) - H(\theta^c))] = \nabla H(\theta^c)^\intercal \E_{\mathbb{P}_n}[\sqrt{n}(\theta - \theta^c)] + \E_{\mathbb{P}_n}[e(\theta) \|\sqrt{n}(\theta - \theta^c)\|],
\end{equation}
where $e(\theta) \rightarrow 0$ if $\theta \rightarrow \theta^c$. The first term in the RHS of \cref{eq4.5} converges weakly to $\mathcal{N}(0, \sigma^2_x)$ by \cref{lemma4.3}. Applying H\"older's inequality to the remainder,
\begin{equation*}
\big|\E_{\mathbb{P}_n}[e(\theta) \|\sqrt{n}(\theta - \theta^c)\|] \big| \le \left(\E_{\mathbb{P}_n}\left[ |e(\theta)|^{\frac{1+\gamma}{\gamma}}\right] \right)^{\frac{\gamma}{1+\gamma}} \left(\E_{\mathbb{P}_n}\left[\|\sqrt{n}(\theta - \theta^c)\|^{1+\gamma} \right] \right)^{\frac{1}{1+\gamma}}.
\end{equation*}
Setting $e(\theta^c) = 0$ does not affect \cref{eq4.5}, so we assume that $e(\cdot)$ is bounded and continuous on $\Theta$ by the continuity of $H$ and the compactness of $\Theta$. From \cref{lemma3.3} we know that $\mathbb{P}_n \Rightarrow \delta_{\theta^c}$ a.s. ($\mathbb{P}^{\N}_{\theta^c}$), thus by \cref{def3.2},
\begin{equation*}
\E_{\mathbb{P}_n}\left[|e(\theta)|^{(1+\gamma)/\gamma}\right] \rightarrow |e(\theta^c)|^{(1+\gamma)/\gamma} = 0 \quad \text{a.s. } (\mathbb{P}^{\N}_{\theta^c}) \quad \text{as } n \rightarrow \infty,
\end{equation*}
which together with \cref{assump4.1} imply that the remainder converges weakly to 0. This proves the the case of mean. For mean-variance, we only need to show that $\sqrt{n} \text{Var}_{\mathbb{P}_n}\{H(\theta)\} \Rightarrow 0$. Note that
\begin{align*}
\sqrt{n} \text{Var}_{\mathbb{P}_n}[H(\theta)] &= \sqrt{n} \text{Var}_{\mathbb{P}_n}[H(\theta) - H(\theta^c)] \le \sqrt{n} \E_{\mathbb{P}_n}\left\{[H(\theta) - H(\theta^c)]^2 \right\} \\
&=\sqrt{n}\E_{\mathbb{P}_n}\left[\left(\nabla H(\theta^c)^\intercal (\theta - \theta^c) + e(\theta) \|\theta - \theta^c\| \right)^2\right]\\
& \le \underbrace{2\sqrt{n} \E_{\mathbb{P}_n}\left[ |\nabla H(\theta^c)^\intercal (\theta - \theta^c) |^2 \right]}_{(*)} +\underbrace{2\sqrt{n} \E_{\mathbb{P}_n}\left\{[e(\theta)]^2 \|\theta - \theta^c\|^2 \right\}}_{(**)},
\end{align*}
where the last inequality follows from $(a+b)^2 \le 2a^2 + 2b^2$. Furthermore, 
\begin{equation*}
(*) \le \frac{2}{\sqrt{n}} \|\nabla H(\theta^c)\|^2 \E_{\mathbb{P}_n}\{\|\sqrt{n}(\theta - \theta^c)\|^2\} \Rightarrow 0 \quad \text{as } n \rightarrow \infty
\end{equation*}
since $\E_{\mathbb{P}_n}\{\|\sqrt{n}(\theta - \theta^c)\|^2\}$ is bounded in probability ($\mathbb{P}^{\N}_{\theta^c}$) by assumption. Similarly, we have $(**) \Rightarrow 0$ by the boundedness of $e(\cdot)$ on $\Theta$.
\end{proof}

\begin{remark}
The proof of \cref{thm4.4} is basically a combination of \cref{lemma4.3} and the Delta method. From definition we know that convergence in total variation implies weak convergence, which together with uniform integrability implies convergence of expectation (see, e.g., \cite{billingsley2013convergence}, Theorem 3.5). This is the main motivation behind \cref{assump4.1}.
\end{remark}

For notational ease, let $\mathbb{P}_X$ denote the distribution of a random variable $X$. Also let $\phi$ and $\Phi$ be the density and cumulative distribution function of $\mathcal{N}(0,1)$, respectively. The forthcoming proof for VaR is based on a series of lemmas presented in the same order as the steps in the heuristic argument: \cref{lemma4.5} copes with the remainder term in Taylor expansion; \cref{lemma4.6} shows that the total variation distance between two distributions will not increase under reasonable mappings; \cref{lemma4.7} closes the final gap between convergence in total variation and convergence of VaR.

\begin{lemma} \label{lemma4.5}
Let $X$ and $Y$ be two random variables, where $X$ and $X+Y$ both have positive densities. Given $\alpha \in (0, 1)$ and $\epsilon \in (0, \min\{\alpha, 1 - \alpha\})$, suppose that $\mathbb{P}(|Y| > \delta) < \epsilon$ for some $\delta>0$. Then,
\begin{equation*}
\text{VaR}^{\alpha - \epsilon}(X) - \delta \le \text{VaR}^\alpha(X+Y) \le \text{VaR}^{\alpha+\epsilon}(X)+\delta.
\end{equation*}
\end{lemma}

\begin{proof}
Since $X$ and $X+Y$ have positive densities, their cumulative distribution functions are continuous and strictly increasing. Thus,
\begin{equation*}
\mathbb{P}(X \le \text{VaR}^\alpha(X)) = \alpha, \quad \mathbb{P}(X+Y \le \text{VaR}^\alpha(X+Y)) = \alpha, \quad \forall \alpha \in (0,1).
\end{equation*}
fully characterizes $\text{VaR}^\alpha(X)$ and $\text{VaR}^\alpha(X+Y)$. The conclusion then follows from the next two observations.
\begin{enumerate}
\item[(i)]
$\mathbb{P}(X+Y \le \text{VaR}^{\alpha - \epsilon}(X) - \delta) \le \alpha:$
\begin{align*}
\text{LHS} &\le \mathbb{P}(X+Y \le \text{VaR}^{\alpha - \epsilon}(X) - \delta, |Y| \le \delta) + \mathbb{P}(|Y| > \delta)\\
&\le \mathbb{P}(X+Y \le \text{VaR}^{\alpha-\epsilon}(X) - \delta, Y \ge -\delta) + \epsilon\\
&\le \mathbb{P}(X \le \text{VaR}^{\alpha - \epsilon}(X)) + \epsilon = \alpha - \epsilon + \epsilon = \alpha.
\end{align*}

\item[(ii)]
$\mathbb{P}(X+Y \le \text{VaR}^{\alpha + \epsilon}(X) + \delta) \ge \alpha:$
\begin{align*}
\text{LHS} &\ge \mathbb{P}(X \le \text{VaR}^{\alpha+\epsilon}(X), Y \le \delta) \\
&\ge 1 - \mathbb{P}(X>\text{VaR}^{\alpha+\epsilon}(X)) - \mathbb{P}(Y > \delta)\\
&\ge \mathbb{P}(X \le \text{VaR}^{\alpha+ \epsilon}(X)) -\epsilon = \alpha + \epsilon - \epsilon = \alpha.
\end{align*}
\end{enumerate}
\end{proof}

\begin{lemma} \label{lemma4.6}
Let $X$ and $Y$ be random variables taking values in a measurable space $(\Omega, \mathcal{F})$. Then, for any measurable function $h: (\Omega, \mathcal{F}) \rightarrow (\tilde{\Omega}, \tilde{\mathcal{F}})$, we have
$$\|\mathbb{P}_{h(X)} - \mathbb{P}_{h(Y)}\|_{\emph{TV}} \le \|\mathbb{P}_X - \mathbb{P}_{Y}\|_{\emph{TV}}.$$
\end{lemma}

\begin{proof}
For any $B \in \tilde{F}$, we have $h^{-1}(B) \in \mathcal{F}$  and
\begin{equation*}
\big|\mathbb{P}_{h(X)}(B) - \mathbb{P}_{h(Y)}(B)\big| = \big|\mathbb{P}_X (h^{-1}(B)) - \mathbb{P}_Y (h^{-1}(B))\big| \le \sup_{A \in \mathcal{F}} |\mathbb{P}_X(A) - \mathbb{P}_Y(A)|,
\end{equation*}
where the last term is $\|\mathbb{P}_X - \mathbb{P}_{Y}\|_{\text{TV}}$. The result follows from taking supremum over $B \in \tilde{\mathcal{F}}$ on both sides.
\end{proof}

\begin{lemma} \label{lemma4.7}
If $X$ is a random variable with positive density, $Y \sim \mathcal{N}(c, \sigma^2)$, and $\|\mathbb{P}_X - \mathbb{P}_Y\|_{\text{TV}} \le \epsilon$, where $\epsilon \in (0, \min\{\alpha, 1 - \alpha\})$ and $\alpha \in (0,1)$, then
\begin{equation*}
\big|\text{VaR}^\alpha (X) - \text{VaR}^\alpha(Y)\big| \le \sigma \max\left\{\Phi^{-1}(\alpha) - \Phi^{-1}(\alpha -\epsilon), \Phi^{-1}(\alpha + \epsilon) - \Phi^{-1}(\alpha)\right\}.
\end{equation*}
\end{lemma}

\begin{proof}
Note that $\text{VaR}^\alpha(Y) = c + \sigma \Phi^{-1}(\alpha)$. Since $\|\mathbb{P}_X - \mathbb{P}_Y\|_{\text{TV}} \le \epsilon$, we have
\begin{align*}
&\mathbb{P}(X \le \text{VaR}_{\alpha+\epsilon}(Y)) \ge \mathbb{P}(Y \le \text{VaR}_{\alpha+\epsilon}(Y)) - \epsilon =\alpha + \epsilon - \epsilon = \alpha,\\
&\mathbb{P}(X \le \text{VaR}_{\alpha-\epsilon}(Y)) \le \mathbb{P}(Y \le \text{VaR}_{\alpha-\epsilon}(Y)) + \epsilon =\alpha - \epsilon + \epsilon = \alpha.
\end{align*}
Hence $\text{VaR}^{\alpha-\epsilon}(Y) \le \text{VaR}^\alpha(X) \le \text{VaR}^{\alpha+\epsilon}(Y)$, and the result follows from the closed forms of $\text{VaR}^{\alpha \pm \epsilon}(Y)$.
\end{proof}

\begin{theorem} \label{thm4.8}
Suppose that \cref{assump3.1} and \cref{assump4.1} hold, and $H$ is differentiable at $\theta^c$. Also assume for $\theta \sim \mathbb{P}_n$ that $\sqrt{n}(\theta - \theta^c)$ and $\sqrt{n}[H(\theta) - H(\theta^c)]$ have positive densities for all $n$ a.s. $(\mathbb{P}^{\N}_{\theta^c})$. Then for any $\alpha \in (0,1)$,
\begin{equation*}
\sqrt{n}\left\{\text{VaR}_{\mathbb{P}_n}^\alpha[H(\theta)] - H(\theta^c) \right\} \Rightarrow \mathcal{N}(\sigma \Phi^{-1}(\alpha), \sigma_x^2).
\end{equation*}
\end{theorem}

\begin{proof}
Since $x$ is fixed, we write $\sigma_x^2$ as $\sigma^2$ for short. Let $Y_n$ denote a random variable with distribution $\mathcal{N}(\nabla H(\theta^c)^\intercal \Delta_n, \sigma^2)$. Our goal is to show that for any $\delta, \epsilon > 0$, there exists $N>0$ such that
\begin{equation} \label{eq4.6}
\mathbb{P}^{\N}_{\theta^c} \left\{\big|\sqrt{n} \left\{\text{VaR}^\alpha_{\mathbb{P}_n}[H(\theta)] - H(\theta^c) \right\}  - \text{VaR}^\alpha(Y_n)\big| > \delta\right\} < \epsilon, \quad \forall n > N.
\end{equation}
By the positive homogeneity and translation invariance of VaR,
\begin{equation*}
\sqrt{n} \left\{\text{VaR}^\alpha_{\mathbb{P}_n}[H(\theta)] - H(\theta^c) \right\} = \text{VaR}^\alpha_{\mathbb{P}_n} \left\{\sqrt{n}[H(\theta) - H(\theta^c)] \right\},
\end{equation*}
which, by first-order Taylor expansion, is equal to
\begin{equation*}
\text{VaR}^\alpha_{\mathbb{P}_n} \left\{X_n(\theta) + e(\theta) \|\sqrt{n}(\theta - \theta^c)\| \right\},
\end{equation*}
where $X_n(\theta):= \nabla H(\theta^c)^\intercal [\sqrt{n}(\theta - \theta^c)]$ and $e(\theta) \rightarrow 0$ if $\theta \rightarrow \theta^c$. Let $\mathbb{P}_{X_n} := \mathbb{P}_n \circ X_n^{-1}$. To show \cref{eq4.6}, we fix a $\delta > 0$ and an $\epsilon > 0$. Note that since $Y_n$ is a normal random variable, there exists an $\eta \in (0, \min\{\alpha, 1-\alpha\})$ such that for all $n$,
\begin{equation} \label{eq4.7}
\big|\text{VaR}^\alpha (Y_n) - \text{VaR}^{\alpha'} (Y_n)\big| = \sigma \big|\Phi^{-1}(\alpha) - \Phi^{-1}(\alpha') \big| < \delta / 3 \quad \text{for} \quad \alpha' = \alpha \pm \eta.
\end{equation}
The rest of the proof is based on constructing the following events.
\begin{enumerate}
\item[(i)]
By the assumption stated in \cref{thm4.8}, we can find an event $E_1 \in \mathbb{B}^{\N}_{\Xi}$ with $\mathbb{P}^{\N}_{\theta^c}(E_1) = 1$ such that on $E_1$, if $\theta \sim \mathbb{P}_n$, then both $\sqrt{n}(\theta - \theta^c)$ and $\sqrt{n}[H(\theta) - H(\theta^c)]$ have positive densities for all $n$.
\item[(ii)]
Since $\E_{\mathbb{P}_n}[\|\sqrt{n}(\theta - \theta^c)\|]$ is bounded in probability ($\mathbb{P}^{\N}_{\theta^c}$) by \cref{assump4.1}, there exists $M_\epsilon > 0$ such that
\begin{equation*}
\mathbb{P}^{\N}_{\theta^c} \left\{\E_{\mathbb{P}_n}[\|\sqrt{n}(\theta -\theta^c)\|] > M_\epsilon \right\} < \epsilon / 3, \quad \forall n.
\end{equation*}
Let $E_{2,n} : = \left\{\E_{\mathbb{P}_n}\|\sqrt{n}(\theta -\theta^c)\|] \le M_\epsilon \right\}$. There exists $M_1 > 0$ on $E_{2,n}$ such that
\begin{equation} \label{eq4.8}
\mathbb{P}_n (\|\sqrt{n}(\theta -\theta^c)\| > M_1) < \eta /2, \quad \forall n
\end{equation}
by Markov's inequality. In addition, from the strong consistency of $\mathbb{P}_n$ and the continuity of $e(\cdot)$, we have
\begin{equation*}
\mathbb{P}_n \left\{|e(\theta)| > \delta / (3M_1) \right\} \rightarrow 0 \quad \text{in probability } (\mathbb{P^{\N}_{\theta^c}}).
\end{equation*}
Therefore, there exists $N_1>0$ such that
\begin{equation} \label{eq4.9}
\mathbb{P}^{\N}_{\theta^c} \left\{ \mathbb{P}_n(|e(\theta)| > \delta / (3M_1)) < \eta / 2\right\}  \ge 1 - \epsilon/3, \quad \forall n > N_1.
\end{equation}
\item[(iii)]
Define $E_{3,n} := \left\{\mathbb{P}_n(|e(\theta)| > \delta / (3M_1)) < \eta / 2 \right\}$ as in the LHS of \cref{eq4.9}. Then, on event $E_{2,n} \cap E_{3,n}$, we have by \cref{eq4.8} that
\begin{align}
&\mathbb{P}_n \left(|e(\theta)|\cdot \|\sqrt{n}(\theta - \theta^c)\| > \delta/3 \right) \notag \\
 \le &\:\mathbb{P}_n \left(\{|e(\theta)| > \delta/(3M_1)\}\cup \{\|\sqrt{n}(\theta - \theta^c)\| > M_1\} \right) < \frac{\eta}{2} +\frac{\eta}{2} = \eta \label{eq4.10}
\end{align}
for all $n > N_1$. 
\item[(iv)]
By \cref{lemma4.7} and the continuity of $\Phi^{-1}$, we can find $\epsilon_1>0$ such that on $E_1$, if $\|\mathbb{P}_{X_n} - \mathbb{P}_{Y_n}\|_{\text{TV}}  \le \epsilon_1$, then
\begin{equation} \label{eq4.11}
\big| \text{VaR}^{\alpha'}_{\mathbb{P}_n}(X_n(\theta)) - \text{VaR}^{\alpha'}(Y_n) \big| < \delta/3 \quad \text{for } \quad \alpha' = \alpha \pm \eta. 
\end{equation}
Meanwhile, since $\|\mathbb{P}_{X_n} - \mathbb{P}_{Y_n}\|_{\text{TV}} \rightarrow 0$ in probability ($\mathbb{P}^{\N}_{\theta^c}$) by \cref{lemma4.6}, there exists $N_2>0 $ such that for the event $E_{4,n} := \left\{\|\mathbb{P}_{X_n} - \mathbb{P}_{Y_n}\|_{\text{TV}} \le \epsilon_1 \right\}$, $\mathbb{P}^{\N}_{\theta^c}(E_{4,n}) \ge 1 - \epsilon/3$ for all $n > N_2$.
\end{enumerate}

Now consider the event $E_n := E_1 \cap E_{2,n} \cap E_{3,n} \cap E_{4,n}$. Take $N:= \max\{N_1, N_2\}$. By a union bound we have $\mathbb{P^{\N}_{\theta^c}}(E_n) \ge 1 - \epsilon/3 - \epsilon/3 - \epsilon/3 = 1 - \epsilon$ for all $n > N$. Moreover, on $E_n$ we have by the definition of $E_1$,  \cref{eq4.10} and \cref{lemma4.5} that,
\begin{equation*}
\text{VaR}_{\mathbb{P}_n}^{\alpha - \eta} (X_n(\theta)) - \frac{\delta}{3} \le \text{VaR}^\alpha_{\mathbb{P}_n}\left\{\sqrt{n}[H(\theta) - H(\theta^c)] \right\} \le \text{VaR}_{\mathbb{P}_n}^{\alpha + \eta} (X_n(\theta)) + \frac{\delta}{3},
\end{equation*}
where by \cref{eq4.11}, 
\begin{equation*}
\text{VaR}^{\alpha - \eta} (Y_n) - \frac{2\delta}{3} \le \text{VaR}^\alpha_{\mathbb{P}_n}\left\{\sqrt{n}[H(\theta) - H(\theta^c)] \right\} \le \text{VaR}^{\alpha + \eta} (Y_n) + \frac{2\delta}{3},
\end{equation*}
and finally by \cref{eq4.7},
\begin{equation*}
\text{VaR}^{\alpha} (Y_n) - \delta \le \text{VaR}^\alpha_{\mathbb{P}_n}\left\{\sqrt{n}[H(\theta) - H(\theta^c)] \right\} \le \text{VaR}^{\alpha } (Y_n) + \delta,
\end{equation*}
which holds for all $n > N$. So \cref{eq4.6} is proved. The conclusion follows from the fact that $\text{VaR}^\alpha(Y_n) = \nabla H(\theta^c)^\intercal \Delta_n + \sigma \Phi^{-1}(\alpha)$ and $\Delta_n \Rightarrow \mathcal{N}(0, [I(\theta^c)]^{-1})$.
\end{proof}

\begin{remark}
Since VaR is not a linear functional of random variables, the remainder term in the Taylor expansion cannot be taken directly outside VaR. Instead, we use \cref{lemma4.5} to control the error caused by ignoring the remainder term. 
\end{remark}

\begin{remark}
Although VaR is not uniformly continuous relative to the total variation metric, the limiting distribution $\mathcal{N}(\Delta_n, [I(\theta^c)]^{-1})$ only varies due to $\Delta_n$, which is a location parameter. This allows us to show in \cref{lemma4.7} that convergence in total variation distance does imply convergence of VaR in the current situation.
\end{remark}

The proof for VaR demonstrates the importance and effectiveness of exploiting the structure of the limiting distribution $\mathcal{N}(\Delta_n, [I(\theta^c)]^{-1})$. In the upcoming proof for CVaR, we continue such exploitation by observing the following properties.

\begin{lemma} \label{lemma4.9}
Suppose $X_n\sim \mathcal{N}(c_n, \sigma^2)$ and there is a constant $ C>0$ such that $|c_n| < C$ for all $n$. Then for any $\epsilon>0$, there exists $M_{C, \epsilon}>0$ such that 
$$\E\left[|X_n| \mathbbm{1}_{\{|X_n|>M_{C, \epsilon}\}} \right] < \epsilon, \quad \forall n.$$ 
\end{lemma}

\begin{proof}
Let $Z \sim \mathcal{N}(C, \sigma^2)$, then there exists $M>0$ such that $\E[Z \mathbbm{1}_{\{Z>M\}}] < \epsilon/2$. It can be verified that this $M$ corresponds to the $\epsilon$ in the lemma.
\end{proof}

\begin{lemma} \label{lemma4.10}
Suppose $X_n \sim \mathcal{N}(c_n, \sigma^2)$ and there is a constant $C>0$ such that $|c_n|<C$ for all $n$. Then for a fixed $\alpha \in (0,1)$ and for any $\epsilon>0$, there exists $\delta_C>0$ such that
\begin{equation*}
\E\left[|X_n| \mathbbm{1}_{\{v_x^n - \delta_C \le X_n \le v_x^n + \delta_C\}}\right] < \epsilon, \quad \forall n,
\end{equation*}
where $v_x^n := \text{VaR}^\alpha(X_n)$.
\end{lemma}

\begin{proof}
Let $Y \sim \mathcal{N}(0, \sigma^2)$ and write $v_y:=\text{VaR}^\alpha(Y)$ for short. Then, for a given $\epsilon>0$, there exists $\delta>0$ such that
$$\E\left[|Y|\mathbbm{1}_{\{v_y - \delta \le Y \le v_y + \delta\}} \right] < \frac{\epsilon}{2}.$$
 Now make $\delta$ smaller (if necessary) such that
 $$\mathbb{P}(v_y - \delta \le Y \le v_y + \delta) < \frac{\epsilon}{2C}.$$
It can be verified that this $\delta$ corresponds to the $\epsilon$ in the lemma.
\end{proof}

\begin{theorem} \label{thm4.11}
Let \cref{assump3.1} and \cref{assump4.1} hold, and $H$ is differentiable at $\theta^c$. Also assume for $\theta \sim \mathbb{P}_n$ that $\sqrt{n}(\theta - \theta^c)$ and $\sqrt{n}[H(\theta) - H(\theta^c)]$ have positive densities for all $n$ a.s. $(\mathbb{P}^{\N}_{\theta^c})$. Then for any $\alpha \in (0,1)$,
\begin{equation*}
\sqrt{n} \left\{\text{CVaR}^\alpha_{\mathbb{P}_n} [H(\theta)] - H(\theta^c) \right\} \Rightarrow \mathcal{N}\left(\frac{\sigma}{1-\alpha}\phi(\Phi^{-1}(\alpha)), \sigma_x^2 \right).
\end{equation*}
\end{theorem}

\begin{proof}
Write $\sigma^2_x$ as $\sigma^2$ for short. Let $X_n(\theta) : = \nabla H(\theta^c)^\intercal [\sqrt{n}(\theta - \theta^c)]$, $\mathbb{P}_{X_n} := \mathbb{P}_n \circ X_n^{-1}$, and let $Y_n$ denote a random variable with distribution $\mathcal{N}(\nabla H(\theta^c)^\intercal\Delta_{n}, \sigma^2)$. Also let $v_x^n := \text{VaR}^\alpha_{\mathbb{P}_n}(X_n(\theta))$ and $v_y^n := \text{VaR}^\alpha(Y_n)$. Note that CVaR is positive homogeneous and translation invariant. By Taylor expansion,
\begin{align*}
&\big|\text{CVaR}^\alpha_{\mathbb{P}_n}\left\{\sqrt{n}[H(\theta) - H(\theta^c)]  \right\} - \text{CVaR}^\alpha_{\mathbb{P}_n} (X_n(\theta)) \big| \\
\le& \: \text{CVaR}^\alpha_{\mathbb{P}_n} \left[|e(\theta)| \cdot \|\sqrt{n}(\theta - \theta^c) \| \right]
\le  \frac{1}{1-\alpha} \E_{\mathbb{P}_n}\left[|e(\theta)| \cdot \|\sqrt{n} (\theta - \theta^c) \| \right],
\end{align*}
which $\Rightarrow 0$ from the proof of \cref{thm4.4}. So it suffices to show that 
$$\text{CVaR}_{\mathbb{P}_n}^\alpha(X_n(\theta)) -  \text{CVaR}^\alpha(Y_n) \rightarrow 0 \quad \text{in probability } (\mathbb{P}^{\N}_{\theta^c}). $$
Fixing a $\delta>0$ and an $\epsilon>0$, we proceed by constructing the following events.
\begin{enumerate}
\item[(i)]
Since $\E[Y_n] = \Delta_{n}$ converges in distribution, it is bounded in probability ($\mathbb{P}^{\N}_{\theta^c}$), and thus for $\epsilon > 0$, there exists $M_1 >0$ such that for the event 
$$E_{1, n} : = \{|\E[Y_n]| \le M_1\},$$ 
$\mathbb{P}^{\N}_{\theta^c}(E_{1,n}) > 1- \epsilon / 4$ for all $n$. By \cref{lemma4.10}, on $E_{1,n}$ we have for $\delta>0$, there exists $\delta_{M_1}>0$ such that
\begin{equation} \label{eq4.12}
\E\left[ |Y_n| \mathbbm{1}_{\{v_y^n - \delta_{M_1} \le Y_n \le v_y^n + \delta_{M_1}\}} \right] < \frac{\delta}{3}.
\end{equation}
By \cref{lemma4.9}, we can find $M_2 > 0$ such that for all $n$, we have on $E_{1,n}$ that
\begin{equation} \label{eq4.13}
\E\left[|Y_n| \mathbbm{1}_{\{|Y_n| > M_2\}} \right] < \frac{\delta}{6}.
\end{equation}

\item[(ii)]
The proof of \cref{thm4.8} implies that $|v_x^n - v_y^n| \rightarrow 0$ in probability ($\mathbb{P}^{\N}_{\theta^c}$), thus we can find $N_1 > 0$ such that the event $E_{2, n} := \{|v_x^n -v_y^n| < \delta_{M_1}\}$ satisifies $\mathbb{P}^{\N}_{\theta^c}(E_{2,n}) > 1- \epsilon / 4$ for all $n > N_1$.

\item[(iii)]
Furthermore, since $\E_{\mathbb{P}_n}[\|\sqrt{n}(\theta - \theta^c)\|^{1+\gamma}]$ is bounded in probability $(\mathbb{P}^{\N}_{\theta^c})$ by \cref{assump4.1}, there exists $M_{\epsilon}>0$  such that
\begin{equation*}
\mathbb{P}^{\N}_{\theta^c}(\E_{\mathbb{P}_n}[\|\sqrt{n}(\theta - \theta^c)\|^{1+\gamma}] > M_\epsilon) < \epsilon/4, \quad \forall n.
\end{equation*}
Let $E_{3,n} := \{\E_{\mathbb{P}_n}\left[\|\sqrt{n}(\theta - \theta^c)\|^{1+\gamma}\right] \le M_{\epsilon}\}$, then by \cref{eqUI} we can find $M_3>0$ such that for all $n$, we have on $E_{3,n}$ that 
\begin{equation} \label{eq4.14}
\E_{\mathbb{P}_n}\left[|X_n(\theta)| \mathbbm{1}_{\{|X_n(\theta)|>M_3 \}} \right] < \frac{\delta}{6}.
\end{equation}

\item[(iv)]
Since $\|\mathbb{P}_{X_n} - \mathbb{P}_{Y_n}\|_{\text{TV}} \rightarrow 0$ in probability ($\mathbb{P}^{\N}_{\theta^c}$) by \cref{lemma4.6}, there exists $N_2>0$ such that for the event $E_{4,n}:=\left\{\|\mathbb{P}_{X_n} - \mathbb{P}_{Y_n}\|_{\text{TV}} \le \delta / 6M \right\}$, we have $\mathbb{P}^{\N}_{\theta^c}(E_{4,n}) > 1 - \epsilon / 4$ for all $n>N_2$. 
\end{enumerate}

Now consider $E_n := E_{1,n} \cap E_{2,n} \cap E_{3,n} \cap E_{4,n}$. Take $M= \max\{M_1, M_2, M_3\}$ and $N = \max\{ N_1, N_2\}$. By a union bound, $\mathbb{P}^{\N}_{\theta^c}(E_n) \ge 1 -\epsilon$ for all $n > N$. Assume without loss of generality that $v_x^n \ge v_y^n$. Then on $E_n$,
\begin{align*}
&\bigg| \E_{\mathbb{P}_n}\left[X_n(\theta) \mathbbm{1}_{\{X_n(\theta) \ge v_x^n\}} \right]  - \E\left[Y_n \mathbbm{1}_{\{Y_n \ge v_y^n\}} \right]\bigg|\\
& \le \underbrace{\bigg| \E_{\mathbb{P}_n}\left[X_n(\theta) \mathbbm{1}_{\{X_n(\theta) \ge v_x^n\}} \right]  - \E\left[Y_n \mathbbm{1}_{\{Y_n \ge v_x^n\}} \right]\bigg|}_{(*)} + \underbrace{\bigg|\E\left[Y_n \mathbbm{1}_{\{v_y^n \le Y_n < v_x^n\}} \right] \bigg|}_{(**)}.
\end{align*}
Note that since $|v_x^n - v_y^n| < \delta_{M_1}$, $(**) \le \delta /3$ by \cref{eq4.12}. Further increase $M$ if necessary so that $M \ge v_x^n$, and for $(*)$ we have
\begin{align*}
\E_{\mathbb{P}_n}\left[X_n(\theta) \mathbbm{1}_{\{X_n(\theta) \ge v_x^n\}} \right] &= \underbrace{ \E_{\mathbb{P}_n}\left[X_n(\theta) \mathbbm{1}_{\{X_n(\theta) > M\}} \right] }_{(\dagger)} + \underbrace{\E_{\mathbb{P}_n}\left[X_n(\theta) \mathbbm{1}_{\{v_x^n \le X_n(\theta) \le M \}} \right]}_{(***)},\\
\E\left[Y_n \mathbbm{1}_{\{Y_n \ge v_x^n\}} \right] &= \underbrace{ \E\left[Y_n \mathbbm{1}_{\{Y_n > M\}} \right]}_{(\dagger\dagger)} + \underbrace{\E\left[Y_n \mathbbm{1}_{\{v_x^n \le Y_n \le M \}} \right]}_{(****)},
\end{align*}
where $|(\dagger)| < \delta/6$ by \cref{eq4.14} and $|(\dagger\dagger)|<\delta / 6$ by \cref{eq4.13}. Define $X^+ := \max(X, 0)$ and $X^- := -\min(X, 0)$, and we have
\begin{align*}
&(***) = \E_{\mathbb{P}_n}\left[X_n^+(\theta) \mathbbm{1}_{\{v_x^n \le X_n(\theta) \le M\}} \right] - \E_{\mathbb{P}_n}\left[X_n^-(\theta) \mathbbm{1}_{\{v_x^n \le X_n(\theta) \le M\}} \right]\\
& = \int_0^\infty \mathbb{P}_n \left(X_n^+(\theta) \mathbbm{1}_{\{v_x^n \le X_n(\theta) \le M\}} > t\right) dt -  \int_0^\infty \mathbb{P}_n\left(X_n^-(\theta) \mathbbm{1}_{\{v_x^n \le X_n(\theta) \le M\}} > t\right) dt\\
& = \int_0^M \mathbbm{P}_n\left(v_x^n \le X_n(\theta) \le M, X_n(\theta) > t\right) dt \\
&\quad - \int_0^M \mathbb{P}_n\left(v_x^n \le X_n(\theta) \le M, X_n(\theta)<-t\right)dt,
\end{align*}
and similarly, $(****)$ can be expressed by
\begin{equation*}
\int_0^M \mathbbm{P}\left(v_x^n \le Y_n \le M, Y_n > t\right) dt - \int_0^M \mathbb{P}\left(v_x^n \le Y_n \le M, Y_n<-t\right)dt.
\end{equation*}
It follows from $\|\mathbb{P}_{X_n} - \mathbb{P}_{Y_n}\|_{\text{TV}} \le \delta / 6M$ that
$$|(***)- (****)| \le M \frac{\delta}{6M} + M \frac{\delta}{6M} = \frac{\delta}{3}.$$
In sum, we have
$$(*) + (**) \le \frac{\delta}{6} + \frac{\delta}{6} + \frac{\delta}{3} + \frac{\delta}{3} = \delta,$$
thus,
$$\mathbb{P}^{\N}_{\theta^c} \left\{\bigg| \E_{\mathbb{P}_n}\left[X_n(\theta) \mathbbm{1}_{\{X_n(\theta) \ge v_x^n\}} \right]  - \E\left[Y_n \mathbbm{1}_{\{Y_n \ge v_y^n\}} \right]\bigg| \le \delta \right\} \ge 1 -\epsilon, \quad \forall n > N,$$
which implies that $\text{CVaR}^\alpha_{\mathbb{P}_n}(X_n(\theta)) - \text{CVaR}^\alpha(Y_n) \rightarrow 0$ in probability ($\mathbb{P}^{\N}_{\theta^c}$). We now conclude that
$$\text{CVaR}_{\mathbb{P}_n}^\alpha\left\{\sqrt{n}[H(\theta) - H(\theta^c)]  \right\} - \text{CVaR}^\alpha (Y_n) \rightarrow 0 \text{ in probability } (\mathbb{P}^{\N}_{\theta^c}).$$
But
\begin{equation*}
\text{CVaR}^\alpha(Y_n) = \nabla H(\theta^c)^\intercal \Delta_{n} + \frac{\sigma}{1-\alpha} \phi(\Phi^{-1}(\alpha)) \Rightarrow \mathcal{N}\left(\frac{\sigma}{1-\alpha}\phi(\Phi^{-1}(\alpha)), \sigma^2 \right),
\end{equation*}
so the proof is complete.
\end{proof}

\begin{remark}
Our proof for CVaR relies on the proof for VaR (see the construction of $E_{2,n}$). Moreover, from the construction of $E_{3,n}$ and \cref{eq4.14} we see that \cref{assump4.1} is critical to bounding the truncated tail expectation of $X_n(\theta)$. This is not surprising since \cref{assump4.1} essentially characterizes  a form of uniform integrability, which is a well-known sufficient condition for bridging the gap between convergence in total variation (or weak convergence) and convergence of expectations. 
\end{remark}

\subsection{Asymptotic normality of optimal values} \label{sec:4.2}
The goal of this section is to establish asymptotic normality of the optimal values
\begin{equation} \label{eq4.16}
\sqrt{n}\left(\min_{x\in \mathcal{X}} \rho_{\, \mathbb{P}_n}[H(x, \theta)] - \min_{x\in \mathcal{X}} H(x, \theta^c) \right).
\end{equation}

Let $C(\mathcal{X})$ denote the Banach space of all continuous functions on a compact set $\mathcal{X}$ equipped with the sup-norm. Also let $\mathcal{C}_{\mathcal{X}}$ denote the Borel $\sigma$-algebra on $C(\mathcal{X})$. A random element \footnote{A random element is a generalization of the concept of random variable to more complicated spaces than $\R$.} is defined as a mapping from $(\Omega, \mathcal{F})$ to $(C(\mathcal{X}), \mathcal{C}_{\mathcal{X}})$, i.e., each realization of a random element is a continuous function in $C(\mathcal{X})$. \cref{def3.2} of weak convergence carries over to this space, except that one need to consider all bounded and continuous functionals on $C(\mathcal{X})$. In words, for $f_n, f \in C(\mathcal{X})$, $f_n \Rightarrow f$ characterizes the weak convergence of continuous random functions. Define
\begin{equation*}
g_n(x):= \sqrt{n}\left\{\rho_{\,\mathbb{P}_n}[H(x, \theta)] - H(x, \theta^c) \right\}.
\end{equation*}
To study the asymptotic distribution of \cref{eq4.16}, we will resort to the following result.
 \begin{lemma}[\cite{shapiro1991asymptotic}, Theorem 3.2] \label{lemma4.12}
 If $\sqrt{n}(f_n - \bar{f}) \Rightarrow Y_x$, where $f_n, \bar{f}$ and $Y_x$ are random elements of $C(\mathcal{X})$, then
\begin{equation} \label{eq4.100}
\sqrt{n}\left(\min_{x \in \mathcal{X}} f_n - \min_{x \in \mathcal{X}} \bar{f}\right) \Rightarrow \min_{x \in S} Y_x \quad \text{as }\: n \rightarrow \infty,
\end{equation}
where $S := \arg\min_{x \in \mathcal{X}} \bar{f}$. 
\end{lemma}

To apply \cref{lemma4.12}, we need to show that (i) $\rho_{\,\mathbb{P}_n}[H(\cdot, \theta)]$ and $H(\cdot, \theta^c)$ are continuous functions on $\mathcal{X}$; (ii) $g_n(\cdot)$ converges weakly to some random element of $C(\mathcal{X})$. In many applications involving empirical distributions (e.g., \cite{kleywegt2002sample,dentcheva2016statistical}), results similar to (ii) can be established using a functional Central Limit Theorem. However, this is not applicable to the Bayesian setting considered in this paper. Instead, we will prove (ii) via two steps. First, we show the weak convergence of $g_n$'s finite-dimensional distributions, i.e., the weak convergence of
$$[g_n(x_1), g_n(x_2), \ldots, g_n(x_k)]$$
for any finite sequence $x_1, x_2, \ldots x_k \in \mathcal{X}$. Then, by Theorem 7.5 in \cite{billingsley2013convergence}, the weak convergence of $g_n$ can be established by checking the following condition
\begin{equation} \label{eq4.18}
\lim_{\delta \rightarrow 0} \limsup_{n \rightarrow \infty} \mathbb{P}^{\N}_{\theta^c}\left(\zeta(g_n, \delta) \ge \epsilon \right) =0, \quad \forall \epsilon>0,
\end{equation}
where $\zeta(f, \delta)$ is the modulus of continuity of $f \in C(\mathcal{X})$ and is defined as
\begin{equation} \label{eq4.19}
\zeta(f, \delta) := \sup_{\|x- x'\|<\delta\atop x, x' \in \mathcal{X}}|f(x) - f(x')|.
\end{equation}
The condition in \cref{eq4.18} is also known as stochastic equicontinuity (s.e.). It guarantees the tightness of $g_n$'s distributions, which implies weak convergence essentially due to the Arzel\`a-Ascoli Theorem (see, e.g., \cite{billingsley2013convergence}, Theorem 7.2).  For more details on weak convergence in the $C$ space, we refer the reader to section 7 in \cite{billingsley2013convergence}.

\begin{theorem} \label{thm4.13}
Suppose that \cref{assump3.1} and \cref{assump4.1} hold. Also assume for $\theta \sim \mathbb{P}_n$ that $\sqrt{n}(\theta - \theta^c)$ and $\sqrt{n}[H(\theta) - H(\theta^c)]$ have positive densities for all $n$ a.s. $(\mathbb{P}^{\N}_{\theta^c})$. Further suppose $\mathcal{X}$ is a compact set, $H$ is continuous on $\mathcal{X} \times \Theta$, and $H(x, \cdot)$ differentiable at $\theta^c$ for all $x \in \mathcal{X}$, where $\nabla_\theta H(\cdot, \theta^c)$ is continuous on $\mathcal{X}$. Then,
\begin{equation*}
\sqrt{n}\left(\min_{x\in \mathcal{X}} \rho_{\,\mathbb{P}_n}[H(x, \theta)] - \min_{x\in \mathcal{X}} H(x, \theta^c) \right) \Rightarrow \min_{x \in S} Y_x,
\end{equation*}
where $S: =\arg\min_{x \in \mathcal{X}} H(x, \theta^c)$ and
\begin{equation*}
Y_x :=
\begin{cases}
\nabla_\theta H(x, \theta^c)^\intercal Z & \mbox{if } \rho=\text{mean / mean-variance}\\
\nabla_\theta H(x, \theta^c)^\intercal Z + \sigma_x \Phi^{-1}(\alpha) & \mbox{if } \rho = \text{VaR}\\
\nabla_\theta H(x, \theta^c)^\intercal Z + \frac{\phi(\Phi^{-1}(\alpha))}{1-\alpha} \sigma_x &\mbox{if } \rho = \text{CVaR}
\end{cases},
\end{equation*}
where $Z$ is a random variable following distribution $\mathcal{N}(0, [I(\theta^c)]^{-1})$.
\end{theorem}

\begin{proof}
{\bf Step 1.} Since $\mathcal{X} \times \Theta$ is compact by the Tychonoff Product Theorem (see, e.g., \cite{royden1988real}, page 245), $H$ is uniformly continuous on $\mathcal{X} \times \Theta$ by the Heine-Cantor Theorem (see, e.g., \cite{rudin1964principles}, Theorem 4.19). Thus, for any $\epsilon>0$, there exists $\delta>0$ such that $|H(x, \theta) - H(x', \theta')|<\epsilon$ as long as $\|(x, \theta) - (x', \theta')\| < \delta$. For the mean, VaR and CVaR formulations, if $\|(x, \theta) - (x', \theta)\| = \|x-x'\|<\delta$, then by \cref{lemma3.8},
\begin{equation*}
\big|\rho_{\,\mathbb{P}_n}\{H(x,\theta)\} - \rho_{\,\mathbb{P}_n}\{H(x', \theta)\}\big| \le \sup_{\theta \in \Theta} \big|H(x, \theta) - H(x',\theta)\big| < \epsilon.
\end{equation*}
So $\rho_{\,\mathbb{P}_n}\{H(\cdot, \theta)\}$ is uniformly continuous on $\mathcal{X}$ for mean, VaR and CVaR. The case of mean-variance follows from the continuity of $H^2$ on $\mathcal{X} \times \Theta$.

{\bf Step 2.} The next step is to show the weak convergence of finite-dimensional distributions. Since mean, VaR and CVaR are linear functionals, we have
$$g_n(x) = \rho_{\, \mathbb{P}_n}\left\{\sqrt{n} [H(x, \theta) - H(x, \theta^c)] \right\}.$$
Fix a finite sequence $x_1, x_2, \ldots, x_k \in \mathcal{X}$. For $[g_n(x_1), g_n(x_2), \ldots, g_n(x_k)]$, we apply Taylor expansion inside the functional $\rho$ for each dimension. The remainder terms also form a $k$-dimensional random vector, which converges in probability to 0 if and only if each dimension does. Thus, the proofs of \cref{thm4.4,thm4.8,thm4.11} can be easily extended to show that each formulation's finite-dimensional distributions converge weakly to that of $Y_x$ defined in the statement of \cref{thm4.13}.
 
{\bf Step 3.} 
By Theorem 7.5 in \cite{billingsley2013convergence}, the proof will be complete if we show that $g_n(\cdot)$ is s.e. (defined as in \cref{eq4.18}) for all four choices of $\rho$, where the specific forms of $Y_x$ follows from \cref{thm4.4,thm4.8,thm4.11}. We now prove s.e. for each choice of $\rho$.

(i) Mean formulation:
By Taylor expansion,
$$g_n(x) = \nabla_\theta H(x, \theta^c)^\intercal \E_{\mathbb{P}_n}[\sqrt{n}(\theta - \theta^c)] + \E_{\mathbb{P}_n}\left[e(x,\theta)\|\sqrt{n}(\theta - \theta^c)\|\right],$$
where it suffices to show s.e. for the two terms in the RHS. Since $\E_{\mathbb{P}_n}[\sqrt{n}(\theta - \theta^c)]$ converges weakly by \cref{assump4.1},  for any $\eta>0$, there exists $ M_\eta > 0$ such that
\begin{equation*}
\mathbb{P}^{\N}_{\theta^c}(\| \E_{\mathbb{P}_n}[\sqrt{n}(\theta -\theta^c)]\| < M_\eta) > 1- \eta, \quad \forall n.
\end{equation*}
Since we assume that $\nabla_\theta H(\cdot, \theta^c)$ is continuous (and hence uniformly continuous) on $\mathcal{X}$, for any $\epsilon>0$, there exists $\delta_\eta>0$ such that
\begin{equation*}
\sup_{\|x- x'\|<\delta_\eta \atop x, x' \in \mathcal{X}} \|\nabla_\theta H(x, \theta^c) - \nabla_\theta H(x', \theta^c)\| < \epsilon/M_\eta.
\end{equation*}
It follows that on the event $\{\|\E_{\mathbb{P}_n}[\sqrt{n}(\theta - \theta^c)]\| < M_\eta \}$, we have
\begin{equation*}
\zeta\left(\nabla_\theta H(x, \theta^c)^\intercal \E_{\mathbb{P}_n}[\sqrt{n}(\theta - \theta^c)], \delta_\eta \right) < \frac{\epsilon}{M_\eta} M_\eta = \epsilon,
\end{equation*}
where $\zeta$ is the modulus of continuity defined in \cref{eq4.19}. Therefore, the first term has the s.e. property. For the second term, we only need to show that
\begin{equation*}
\sup_{x \in \mathcal{X}} \big|\E_{\mathbb{P}_n}\left[e(x, \theta) \|\sqrt{n}(\theta - \theta^c)\|\right]\big| \Rightarrow 0.
\end{equation*}
Since
\begin{equation*}
\sup_{x \in \mathcal{X}} \big|\E_{\mathbb{P}_n}[e(x, \theta) \|\sqrt{n}(\theta - \theta^c)\|]\big| \le \E_{\mathbb{P}_n}\left[\sup_{x \in \mathcal{X}} |e(x, \theta)| \|\sqrt{n}(\theta - \theta^c)\|\right],
\end{equation*}
it suffices to show for $\theta \sim \mathbb{P}_n$ that $\sup_{x \in \mathcal{X}} |e(x, \theta)| \Rightarrow 0$ a.s. ($\mathbb{P}^{\N}_{\theta^c}$). However, the continuity of $e$ on $\mathcal{X} \times \Theta$ implies that $\sup_{x\in \mathcal{X}} |e(x, \cdot)|$ is continuous on $\Theta$. Setting $e(x, \theta^c) = 0$ for all $x\in \mathcal{X}$ does not affect the Taylor expansion, so for $\theta \sim \mathbb{P}_n$,
$$\sup_{x \in \mathcal{X}}|e(x, \theta)| \Rightarrow \sup_{x \in \mathcal{X}}|e(x, \theta^c)| = 0 \quad \text{a.s. }(\mathbb{P}^{\N}_{\theta^c}),$$
and the rest follows from the proof of \cref{thm4.4}.

(ii) Mean-variance formulation:
Since $\nabla_\theta H(\cdot, \theta^c)$ is bounded on $\mathcal{X}$ and $e$ is bounded on $\mathcal{X} \times \Theta$, it follows from the proof of \cref{thm4.4} that
$$\sup_{x \in \mathcal{X}} \sqrt{n} \text{Var}_{\mathbb{P}_n}[H(x, \theta)] \Rightarrow 0,$$
which implies the s.e. for mean-variance.

(iii) VaR formulation:
Recall that the proof of \cref{thm4.8} is based on bounding 
\begin{align*}
(\dagger) &= g_n(x) - \text{VaR}_{\mathbb{P}_n}^{\alpha \pm \epsilon_1}\{\nabla_\theta H(x, \theta^c)^\intercal [\sqrt{n}(\theta - \theta^c)] \},
\end{align*}
by \cref{eq4.10} and \cref{lemma4.5} for some $\epsilon_1>0$, and
\begin{align*}
(\dagger\dagger) &=\text{VaR}_{\mathbb{P}_n}^{\alpha \pm \epsilon_1}\{\nabla_\theta H(x,\theta^c)^\intercal [\sqrt{n} (\theta -\theta^c)] \} - \text{VaR}_{\mathbb{P}_n}^{\alpha \pm \epsilon_1}\{\mathcal{N}(\nabla H(x,\theta^c)^\intercal\Delta_{n}, \sigma_x^2) \},
\end{align*}
by \cref{eq4.11}, where the bound on $|(\dagger)|$ depends on $x$ via $e(x, \theta)$, and the bound on $|(\dagger\dagger)|$ does not depend on $x$ due to \cref{lemma4.6}. Since we have for $\theta \sim \mathbb{P}_n$ that $\sup_{x\in \mathcal{X}} |e(x, \theta)| \Rightarrow 0$ a.s. ($\mathbb{P}^{\N}_{\theta^c}$), following the proof of \cref{thm4.8} yields that
\begin{equation*}
\sup_{x \in \mathcal{X}} \big|g_n(x)  - \text{VaR}_{\mathbb{P}_n}^\alpha\{\mathcal{N}(\nabla H(x,\theta^c)^\intercal\Delta_{n}, \sigma_x^2)\big| \rightarrow 0 \quad \text{in probability } (\mathbb{P}^{\N}_{\theta^c}).
\end{equation*}
But we know
$$\text{VaR}^\alpha\{\mathcal{N}(\nabla H(x,\theta^c)^\intercal\Delta_{n}, \sigma_x^2)\} = \nabla_\theta H(x, \theta^c)^\intercal \Delta_{n} + \sigma_x \Phi^{-1}(\alpha)$$
has s.e. since $\nabla_\theta H(\cdot, \theta^c)$ and $\sigma_x^2$ are uniformly continuous on $\mathcal{X}$ and $\Delta_{n}$ converges in distribution. So the case of VaR is proved.

(iv) CVaR formulation:
Since
\begin{equation*}
g_n(x) = \text{CVaR}_{\mathbb{P}_n}^\alpha\{\nabla_\theta H(x, \theta^c)^\intercal [\sqrt{n}(\theta - \theta^c)] + e(x, \theta) \|\sqrt{n}(\theta - \theta^c) \| \},
\end{equation*}
we have
\begin{align*}
\zeta(g_n, \delta) \le& \sup_{\|x- x'\|<\delta\atop x, x' \in \mathcal{X}} \frac{1}{1-\alpha} \left\{\E_{\mathbb{P}_n} \left[\|\nabla_\theta H(x, \theta^c) - \nabla_\theta H(x', \theta^c) \| \|\sqrt{n} (\theta -\theta^c)\| \right] \right.\\
& \quad \left. + \E_{\mathbb{P}_n}\left[|e(x, \theta) - e(x', \theta)| \|\sqrt{n} (\theta -\theta^c)\| \right] \right\},
\end{align*}
and the rest follows from the proof for mean.
\end{proof}

\section{Interpretation of Bayesian risk optimization} \label{sec:5}
We now interpret BRO based on the asymptotic normality results established in \cref{sec:4.2}. Following the notations in \cref{thm4.13}, let $Z$ denote a random variable with distribution $\mathcal{N}(0, [I(\theta^c)]^{-1})$. We write $\sigma_x$ as $\sigma(\cdot)$ to emphasize that it is a function of $x$. Taking the VaR formulation as an example, from the proof of \cref{thm4.13} (for VaR) we have the following weak convergence result in the $C$ space.
\begin{equation*}
\sqrt{n} \left\{\text{VaR}_{\mathbb{P}_n}^\alpha[H(\cdot,\theta)] - H(\cdot, \theta^c) \right\} \Rightarrow \nabla_\theta H(\cdot, \theta^c)^\intercal Z +  \Phi^{-1}(\alpha) \sigma(\cdot).
\end{equation*}
This can be rewritten as
\begin{equation} \label{eq5.1}
\text{VaR}_{\mathbb{P}_n}^\alpha[H(\cdot, \theta)] \stackrel{\mathcal{D}}{=}  H(\cdot, \theta^c) + \frac{\nabla_\theta H(\cdot, \theta^c)^\intercal Z}{\sqrt{n}} + \Phi^{-1}(\alpha) \frac{\sigma(\cdot)}{\sqrt{n}} + o_p\left(\frac{1}{\sqrt{n}} \right),
\end{equation}
where ``$\stackrel{\mathcal{D}}{=}$'' means ``is distributionally equivalent to'', and $o_p(1/\sqrt{n})$ stands for a term whose product with $\sqrt{n}$ converges to 0 in probability $(\mathbb{P}^{\N}_{\theta^c})$ uniformly in $x$. The LHS of \cref{eq5.1} is the VaR objective we propose to minimize, and the RHS can be viewed as the sum of the true objective $H(\cdot, \theta^c)$ and some error terms. Compared with the mean formulation, whose objective can be written as
\begin{equation} \label{eq5.2}
\E_{\mathbb{P}_n}[H(\cdot,\theta)] \stackrel{\mathcal{D}}{=}  H(\cdot, \theta^c) + \frac{\nabla_\theta H(\cdot, \theta^c)^\intercal Z}{\sqrt{n}} + o_p\left(\frac{1}{\sqrt{n}} \right),
\end{equation}
we see that \cref{eq5.1} has an extra deterministic bias term $\Phi^{-1}(\alpha) \sigma(\cdot) / \sqrt{n}$ that vanishes as $n \rightarrow \infty$. Combining \cref{eq5.1} and \cref{eq5.2}, we have
\begin{equation*}
\text{VaR}_{\mathbb{P}_n}^\alpha[H(\cdot,\theta)] \stackrel{\mathcal{D}}{=}  \E_{\mathbb{P}_n}[H(\cdot, \theta)] + \Phi^{-1}(\alpha) \frac{\sigma(\cdot)}{\sqrt{n}} + o_p\left(\frac{1}{\sqrt{n}} \right).
\end{equation*}
Therefore, the VaR formulation's objective approximately equals a weighted sum of posterior mean and a bias $\sigma(\cdot) / \sqrt{n}$, where the weight of the bias is $\Phi^{-1}(\alpha)$. Although the bias diminishes as $n \rightarrow \infty$, it has an undeniable impact on the VaR objective when $n$ is small. In particular, if $n$ is not too large (e.g. 20) and $\alpha$ is close to 1 (e.g. 99\%), it is possible for the bias term to dominate $\E_{\mathbb{P}_n}[H(x,\theta)]$ and we are close to solving $\min_{x\in \mathcal{X}} \sigma_x /\sqrt{n}$. Similarly, the CVaR objective can be rewritten as
\begin{equation*}
\text{CVaR}_{\mathbb{P}_n}^\alpha\{H(\cdot, \theta)\} \stackrel{\mathcal{D}}{=}  \E_{\mathbb{P}_n}[H(\cdot, \theta)] + \frac{\phi(\Phi^{-1}(\alpha))} {1 - \alpha} \frac{\sigma(\cdot)}{\sqrt{n}} + o_p\left(\frac{1}{\sqrt{n}} \right).
\end{equation*}
For the mean-variance formulation, by imposing appropriate conditions of uniform integrability, it can be shown that the variance satisfies
\begin{equation*}
\text{Var}_{\mathbb{P}_n}[H(\cdot, \theta)] = \frac{1}{n}\text{Var}_{\mathbb{P}_n}\{\sqrt{n}[H(\cdot, \theta) - H(\cdot, \theta^c)]\} \stackrel{\mathcal{D}}{=}  \frac{\sigma^2(\cdot)}{n} + o_p\left(\frac{1}{n} \right).
\end{equation*}
Thus, the objective functions of the mean-variance, VaR and CVaR formulations are all approximately equivalent to a weighted sum of the mean objective and $\sigma(\cdot) / \sqrt{n}$ (or $(\sigma(\cdot) / \sqrt{n})^2$), where the weight of $\sigma(\cdot) / \sqrt{n}$ is controlled by $\alpha$ in the VaR and CVaR formulations, and the weight of $(\sigma(\cdot) / \sqrt{n})^2$ is controlled by the constant $w$ in the mean-variance formulation.

At this point, one naturally wonders the implication of minimizing $\sigma(\cdot)/\sqrt{n}$. For a fixed $x\in \mathcal{X}$, \cref{thm4.4,thm4.8,thm4.11} allow the following asymptotical valid $100(1-\beta)\%$ confidence intervals (CIs) of $H(x, \theta^c)$ (in the form of center $\pm$ half-width).
\begin{enumerate}
\item[(i)]
The mean and mean-variance formulation
\begin{equation*}
\big(\E_{\mathbb{P}_n}[H(x, \theta)] + w\text{Var}_{\mathbb{P}_n}[H(x,\theta)]\big) \pm z_{1-\frac{\beta}{2}} \frac{\sigma_x}{\sqrt{n}},
\end{equation*}
\item[(ii)]
The VaR formulation
\begin{equation*}
\left(\text{VaR}_{\mathbb{P}_n}^{\alpha}[H(x,\theta)] - \frac{\Phi^{-1}(\alpha) \sigma_x}{\sqrt{n}} \right) \pm z_{1 - \frac{\beta}{2}} \frac{\sigma_x}{\sqrt{n}},
\end{equation*}
\item[(iii)]
The CVaR formulation
\begin{equation*}
\left(\text{CVaR}_{\mathbb{P}_n}^{\alpha}[H(x,\theta)] - \frac{\phi(\Phi^{-1}(\alpha)) \sigma_x}{(1-\alpha)\sqrt{n}}\right) \pm z_{1 - \frac{\beta}{2}} \frac{\sigma_x}{\sqrt{n}},
\end{equation*}
\end{enumerate}
 where $z_{1-\beta}$ denotes the $(1-\beta)$-quantile of a standard normal distribution. Observe that $\sigma_x/\sqrt{n}$ is exactly proportional to the half-width of the CI, where narrower CI implies higher accuracy of estimating the true performance $H(x, \theta^c)$, while a wider CI indicates higher risk due to less confidence about how a solution actually performs. {\emph{In other words, BRO is essentially seeking a tradeoff between posterior expected performance and the robustness in actual performance.}} It is also interesting to notice that $\sigma_x$ depends on $\nabla_\theta H(x, \theta^c)$ and $I(\theta^c)$, where $\nabla_\theta H(x, \theta^c)$ is the sensitivity of the function $H$ to the perturbation of the parameter $\theta^c$ (i.e. the distributional uncertainty), while $I(\theta^c)$ is the Fisher information that captures the amount of information a data sample carries about the true parameter.

\section{Concluding remarks} \label{sec:6}
We formally propose a framework of Bayesian risk optimization (BRO) for data-driven stochastic optimization problems, and study the implications of BRO by establishing a series of consistency and asymptotic normality results. The analysis on asymptotics leads to an important insight: BRO explicitly seeks a tradeoff between posterior mean performance and the risk in a solution's actual performance. A question of practical interest is how to choose the weight  in the mean-variance formulation or the risk level in the VaR and CVaR formulations, since these two parameters control the balance between posterior expected performance and robustness. In addition, our proofs assume compactness of the parameter space, but it is worth studying more general cases since many priors are not supported on compact sets. It is also interesting to consider a nonparametric setting with a prior of Dirichlet process, though the associated asymptotics could be much more complicated.

\section*{Acknowledgement} We would like to thank Professor Alexander Shapiro for making valuable comments on our paper. We also thank the editor and the anonymous reviewers for significantly improving the presentation of the paper.

\bibliographystyle{ieeetr}
\bibliography{risk_formulations}
\end{document}
%